\documentclass[10pt]{article}
\usepackage{amsmath, amsthm, amssymb, amsfonts, mathrsfs, mathtools, upgreek}
\usepackage{graphicx}
\usepackage{makeidx}
\usepackage{tikz-network}
\usepackage[left=2cm,right=2cm,top=2cm,bottom=2cm]{geometry}
\usepackage{caption}
\usepackage{subcaption}
\usepackage[section]{placeins}
\usepackage{enumitem}
\usepackage{tikz}
\usepackage{stmaryrd}
\usepackage{upquote}
\usepackage{authblk}

\usepackage{multicol}

\usepackage{hyperref}
\hypersetup{colorlinks=true, citecolor={blue}, linkcolor=blue, filecolor=magenta, urlcolor=cyan}
\usetikzlibrary{arrows.meta, positioning, decorations.pathreplacing}
\usepackage{cleveref}
\newcommand{\compconj}[1]{\overline{#1}}

\newtheorem{theorem}{Theorem}[section]

\newtheorem{lemma}[theorem]{Lemma}
\newtheorem{remark}[theorem]{Remark}

\newtheorem{corollary}[theorem]{Corollary}
\newtheorem{proposition}[theorem]{Proposition}

\theoremstyle{definition}
\newtheorem{defi}[theorem]{Definition}

\numberwithin{equation}{section}
\newcommand{\ip}[2]{\left\langle #1, #2 \right\rangle}
\newcommand{\divides}{\mid}

\newcommand{\homo}{\operatorname{Hom}}
\newcommand{\gl}{\operatorname{GL}}
\newcommand{\aut}{\operatorname{Aut}}

\newcommand{\cay}{\operatorname{Cay}}

\newcommand{\iu}{\mathbf{i}}

\newcommand{\lcm}{\operatorname{lcm}}

\newcommand{\Zl}{\mathbb{Z}}
\newcommand{\Rl}{\mathbb{R}}
\newcommand{\Ql}{\mathbb{Q}}

\newcommand{\Cl}{\mathbb{C}}

\newcommand{\eu}{\mathbf{e}_u}
\newcommand{\ev}{\mathbf{e}_v}

\newcommand{\evt}{\mathbf{e}_v^t}
\newcommand{\ld}{\lambda}

\title{\textbf{A complete classification of perfect state transfer in Grover walks on Cayley graphs over dihedral groups}
}
\author[1,*]{Koushik Bhakta}
\author[1]{Bikash Bhattacharjya}
\author[2]{Xiwang Cao}

\affil[1]{\small{Department of Mathematics, Indian Institute of Technology Guwahati, India}}
\affil[2]{\small{School of Mathematical Sciences, Nanjing University of Aeronautics and Astronautics, China}}

\date{}
\begin{document}
	\maketitle
	\begingroup
	\renewcommand{\thefootnote}{}
	\footnote{* Corresponding author.\\E-mail addresses:  b.koushik@iitg.ac.in (K. Bhakta), b.bikash@iitg.ac.in (B. Bhattacharjya), xwcao@nuaa.edu.cn (X. Cao).}
	\addtocounter{footnote}{-1}
	\endgroup
	
	\vspace{-1cm}
	
	\begin{center}{\textbf{Abstract}}\end{center}
	\noindent  
	We give a complete classification of Cayley graphs over dihedral groups that exhibit perfect state transfer between vertex states in Grover walks. For $n\ge 3$, let $D_n:=\langle a,b : a^n=b^2=1,\ bab=a^{-1}\rangle$ be the dihedral group of order $2n$, and let $S\subseteq D_n\setminus\{1\}$ satisfying $S=S^{-1}$ and $\langle S\rangle =D_n$. We show that the Cayley graph $\operatorname{Cay}(D_n,S)$ exhibits perfect state transfer if and only if $(D_n,S)$ is equivalent to one of the following pairs:  $(D_n,\{b,ba\})$;  $(D_n,\{a,a^{-1},ba,ba^{-1}\})$, where $4$ does not divide $n$; or   $(D_{2m},\{b,ba^{m},ba,ba^{m+1}\})$, where  $m$ is odd. 	Our proof combines the representation theory of $D_n$ with tools from algebraic number theory.
	
	\vspace*{.5cm}
	
	
	
	\section{Introduction}
	Quantum walks provide a framework for quantum algorithms~\cite{quantum_algorithm}, quantum cryptography~\cite{cryptography}, quantum simulations~\cite{first_grover}, and others. They are broadly classified into continuous-time~\cite{state} and discrete-time~\cite{godsil_dqw} models.  Several frameworks for continuous-time and discrete-time quantum walks have been introduced and extensively studied in the literature~\cite{banerjee2,state,godsil_pure_state,guo_dqw1,bethetree,sin_twodimesion,sarkar,peter,zhan3}. A fundamental question is whether a quantum walk can transfer a quantum state perfectly between distinct vertices, a phenomenon known as perfect state transfer~\cite{christandl}. In this paper, we study this question for Grover walks~\cite{hig3}, a widely studied model of discrete-time quantum walks. Most previous studies of perfect state transfer in Grover walks have focused on Cayley graphs over abelian groups (see~\cite{bhakta2,bhakta3,kubota2,zhan1}). This leads naturally to the corresponding problem for non-abelian groups. Sarkar and Adhikari~\cite{sarkar1} studied periodicity and localization on Cayley graphs over dihedral groups using generalized Grover coins. Here, we give a complete classification of Cayley graphs over dihedral groups that exhibit perfect state transfer in Grover walks. We determine all possible connection sets up to automorphisms of dihedral groups, giving an explicit classification for an infinite family of non-abelian groups (see Theorem~\ref{main_pst_dn}).

	\subsection{Grover walk}
	Let $G$ be a finite simple graph without isolated vertices, with vertex set $V(G)$ and edge set $E(G)$. We denote the set of symmetric arcs by $\mathcal{A}(G):=\big\{(u,v),(v,u):\{u,v\}\in E(G)\big\}$. For an arc $a\in\mathcal{A}(G)$ given by the ordered pair $(u,v)$, the \emph{origin} $o(a)$ and the \emph{terminus} $t(a)$ are defined by $o(a)=u$ and $t(a)=v$, respectively.

	For any matrix $M$, let $M_{x,y}$ denote its entry at row $x$ and column $y$. We first introduce the matrices defining the Grover walk. The \emph{shift matrix} $R:=R(G)\in\Cl^{\mathcal{A}(G)\times\mathcal{A}(G)}$ of $G$ is defined by $R_{a,b}=\delta_{a^{-1}}(b)$, where $\delta$ is the Kronecker delta function. The \emph{boundary matrix} $N:=N(G)\in\Cl^{V(G)\times\mathcal{A}(G)}$ of $G$ is defined by $N_{u,a}=\frac{1}{\sqrt{\deg u}}\delta_{u}(t(a))$, where $\deg u$ denotes the degree of the vertex $u$ in $G$. The \emph{transition matrix} $U:=U(G)\in\Cl^{\mathcal{A}(G)\times\mathcal{A}(G)}$ of $G$ is defined by $U=R(2{N^*}N-I)$, where $I$ is the identity matrix and $N^*$ is the conjugate transpose of $N$. 
	
	We refer to \cite{kubota2} for further details about these matrices and for a graphical interpretation of the action of the transition matrix. The discrete-time quantum walk defined by the unitary matrix $U$ is called the \emph{Grover walk} (see \cite{hig_grover}). The Grover walk can also be viewed as a special case of a bipartite walk \cite{chen}.
	
	The spectrum of $U$ is determined by a smaller matrix known as the \emph{discriminant matrix}, denoted $P:=P(G)\in\Cl^{V(G)\times V(G)}$, defined as $P=NRN^*$.  The \emph{adjacency matrix} $A:=A(G)\in\Cl^{V(G)\times V(G)}$ of a graph $G$ is defined by $A_{u,v}=1$ if $\{u,v\}\in E(G)$ and $0$ otherwise. In the case of a regular graph, the discriminant matrix and the adjacency matrix are directly correlated.
	\begin{lemma}[{\cite[Lemma~2.1]{kubota2}}]\label{reg}
		If $G$ is a $k$-regular graph, then $P=\frac{1}{k}A$. 
	\end{lemma}
	For any nonnegative integer $k$, we write $\{a\}^k$ for the multiset $\{a,\ldots,a\}$ in which the element $a$ appears exactly $k$ times. Throughout the paper, the imaginary unit $\sqrt{-1}$ is denoted by $\iu$. 
	\begin{theorem} [{\cite[Proposition~1]{hig3}}] \label{ev_grover}
		Let $\mu_1, \ldots,\mu_n$ be the eigenvalues of the discriminant of a connected graph $G$. Then the spectrum of the transition matrix $U$ of $G$ is given by
		\[\left\{e^{\pm \iu \arccos(\mu_j)}:1\leq j\leq n\right\}\cup \{1\}^{b_1} \cup \{ -1\}^{b_1-1+1_B},\] 
		where $b_1=|E(G)|-|V(G)|+1$, and $1_B=1$ or $0$ according as $G$ is bipartite or not.	
	\end{theorem}
	The Grover walk on a graph is called \emph{periodic} if there exists a positive integer $\tau$ such that  $U^\tau=I$. For simplicity, we say that the graph itself is periodic in this case. If $\tau$ is the smallest positive integer satisfying $U^\tau =I$, then the \emph{period} of the graph is $\tau$, and the graph is called \emph{$\tau$-periodic}. Since $U$ is diagonalizable, the period  of a graph can be easily found from the eigenvalues of $U$.
	\begin{lemma}\label{periodic}
		Let $\eta_1, \ldots, \eta_n$ be the  eigenvalues of the transition matrix of a graph $G$. Then $G$ is periodic if and only if every eigenvalue $\eta_j$ of $U$ is a root of unity. In that case, if $k_j$ is the order of $\eta_j$, the period is $\lcm(k_1, \ldots, k_n)$.
	\end{lemma}
	Yoshie~\cite{odd_periodic} proved that a connected graph $G$ is periodic with odd period $k$ if and only if $G$ is the cycle $C_k$. For even periods, Kubota et al. later obtained partial results.
	\begin{lemma}[{\cite[Theorem~1.1]{regular}}]\label{kubota_regular}
		Let $G$ be a $k$-regular connected graph with $k\geq2$, and let $\ell$ be an odd integer. If $G$ is a $2\ell$-periodic, then either $G$ is the cycle graph $C_{2\ell}$ or $k=3$ and $\ell$ is multiple of three.
	\end{lemma}
	A vector $\Phi\in\Cl^{\mathcal{A}(G)}$ is called a \emph{state} if it has unit Euclidean norm. We say \emph{perfect state transfer} occurs from a state $\Phi_1$ to another state $\Phi_2$ at time $\tau$ if there exists a complex number $\gamma$ with $|\gamma|=1$ such that $U^\tau\Phi_1=\gamma\Phi_2$. We are mainly interested in transfer between states that are localized at individual vertices of the graph. For a vertex $u$, define $\Phi_u=N^*\eu$, where $(\eu)_x=\delta_{u}(x)$. The state $N^*\eu$ is called a \emph{vertex-type} state. For a motivation behind this definition and a geometric interpretation of vertex-type states, we refer the reader to \cite{kubota1}.
	\begin{defi}
		A graph is said to exhibit perfect state transfer from vertex $u$ to another vertex $v$ at time $\tau$ if there exists a complex number $\gamma$ with $|\gamma|=1$ such that $U^\tau \Phi_u=\gamma\Phi_v$.
	\end{defi}
	Perfect state transfer in Grover walks can be analyzed effectively using Chebyshev polynomials. The \emph{Chebyshev polynomial of the first kind}, denoted by $T_n(x)$, defined by the recurrence relation: $T_0(x)=1$, $T_1(x)=x$, and $T_n(x)=2xT_{n-1}(x)-T_{n-2}(x)$ for $n\geq 2$. It is well known that $T_n(\cos x)=\cos nx$. Kubota and Yoshino~\cite{kubota2} proved that perfect state transfer occurs from $u$ to $v$ if and only if 	$T_\tau(P)\mathbf{e}_u = \gamma \mathbf{e}_v$ for some $\gamma \in \{ -1,1\}$. 	Later, Guo and Schmeits~\cite{guo_sch} refined this result by showing that perfect state transfer occurs if and only if 	$T_\tau(P)\mathbf{e}_u = \mathbf{e}_v$.
	\begin{lemma}[{\cite[Theorem~6.5]{kubota2}}]\label{defpst}
		Let $u$ and $v$ be two vertices of a graph $G$. Then perfect state transfer occurs from $u$ to $v$ at time $\tau$ if and only if $ T_\tau(P)\eu=\ev$. 
	\end{lemma}
	We say a graph $G$ exhibits {perfect state transfer} if there exist distinct vertices $u$ and $v$ in $G$ and a positive integer $\tau$ such that perfect state transfer occurs from $u$ to $v$ at time $\tau$. In \cite[Lemma~3.2]{kubota1}, Kubota and Segawa proved that $\|T_n(P)\eu\|\leq 1$ for each $u$. Thus $T_n(P)\eu=\ev$ if and only if $T_n(P)_{u,v}=1$. Since $P$ is symmetric, so is $T_n(P)$, and hence $T_n(P)_{u,v}=T_n(P)_{v,u}$. Consequently, perfect state transfer from $u$ to $v$ occurs if and only if it occurs from $v$ to $u$. We may therefore simply say that perfect state transfer occurs between $u$ and $v$. 
	\subsection{Main results}
	Let $\Gamma$ be a finite group with the identity element $1$, and let $S$ be a non-empty subset of $\Gamma$. The \emph{Cayley graph} $\cay(\Gamma,S)$ has vertex set $\Gamma$, with two vertices $u$ and $v$ adjacent if $uv^{-1}\in S$. Throughout the paper, we assume that $1\notin S$, $S=S^{-1}:=\{s^{-1}:s\in S\}$, and $\langle S\rangle=\Gamma$, so that $\cay(\Gamma,S)$ is a simple connected graph. When $\Gamma=\Zl_n$, the resulting Cayley graph is called a \emph{circulant graph}.
	
	Perfect state transfer in Grover walks has been studied extensively on Cayley graphs over abelian groups. Zhan~\cite{zhan1} constructed an infinite family of circulant graphs exhibiting perfect state transfer. Kubota and Yoshino~\cite{kubota2} completely characterized the circulant graphs of valency at most $4$ that exhibit perfect state transfer. More recently, Bhakta and Bhattacharjya~\cite{bhakta_abelian} obtained a complete classification of Cayley graphs over abelian groups exhibiting perfect state transfer, in terms of the underlying group and the connection set.

	These results lead naturally to the study of perfect state transfer on Cayley graphs over non-abelian groups. In this paper, we consider the dihedral group. For $n\geq3$, the dihedral group of order $2n$ is
	\[
	D_n=\left\langle a,b \mid a^n=b^2=1,\; bab=a^{-1}\right\rangle.
	\]
	Dihedral groups provide a natural setting for extending the abelian theory, as their irreducible representations have degrees $1$ or $2$. We use representation theory and algebraic number theory to determine all connection sets that allow perfect state transfer. The Cayley graph $\cay(\Gamma,S)$ is called \emph{normal} if $S$ is a union of conjugacy classes of $\Gamma$. This condition always holds for abelian groups, but we impose no such restriction in our classification. To state our classification, we say that two pairs $(\Gamma,S)$ and $(\Gamma',S')$ are \emph{equivalent} if there exists a group isomorphism $\varphi:\Gamma\to\Gamma'$ such that $\varphi(S)=S'$. In particular, $(D_n,S)$ and $(D_n,T)$ are equivalent if $\varphi(S)=T$ for some automorphism $\varphi$ of $D_n$.   
	\begin{theorem}\label{main_pst_dn}
		Let $n\geq3$ and let $S\subseteq D_n\setminus\{1\}$ satisfying $S=S^{-1}$ and $\langle S\rangle=D_n$.
		Then $\mathrm{Cay}(D_n,S)$ exhibits perfect state transfer if and only if $(D_n,S)$ is equivalent to exactly one of the following pairs.
		\begin{enumerate}[label=\textup{(\roman*)}]
			\item $(D_n,\{b,ba\})$;
			\item $(D_n,\{a,a^{-1},ba,ba^{-1}\})$, where $4$ does not divide $n$; or
			\item $(D_{2m},\{b,ba^{m},ba,ba^{m+1}\})$, where $m$ is odd.
		\end{enumerate}
	\end{theorem}
	\begin{table}[ht]
		\centering
		\small
		\setlength{\tabcolsep}{10pt}
		\renewcommand{\arraystretch}{1.5}
		\begin{tabular}{@{}cccc@{}}
			\hline
			Connection set $S$ & Condition
			& Minimum time & Transfer pairs\\
			\hline
			$\{b,ba\}$
			& $n$ even
			& $n$
			&$\{a^r,a^{r+n/2}\}$ and $\{ba^r,ba^{r+n/2}\}$\\[2pt]
			
			$\{b,ba\}$
			& $n$ odd
			& $n$
			& $\{a^r,ba^{r+(n+1)/2}\}$\\[2pt]
			
			$\{a,a^{-1},ba,ba^{-1}\}$
			& $4\nmid n$
			& $\lcm(n,2)$
			& $\{a^r,ba^r\}$\\[2pt]
			
			$\{b,ba^m,ba,ba^{m+1}\}$
			& $n=2m$, $m$ odd
			& $2m$
			&  $\{a^r,a^{r+m}\}$ and $\{ba^r,ba^{r+m}\}$\\
			\hline
		\end{tabular}
		\caption{Complete classification of perfect state transfer on $\cay(D_n,S)$.}
		\label{pst_pairs_dn}
	\end{table}
	For each family in Theorem~\ref{main_pst_dn}, we also determine the minimum transfer time and all transfer pairs. The following proposition summarizes these in Table~\ref{pst_pairs_dn}.
	\begin{proposition}\label{pst_transfer_pairs}
		For each pair $(D_n,S)$ listed in Theorem~\ref{main_pst_dn}, the minimum time and all pairs of vertices exhibiting perfect state transfer are given in Table~\ref{pst_pairs_dn}, where $r\in\Zl_n$.
	\end{proposition}

	Theorem~\ref{main_pst_dn} gives a classification in terms of the group and the connection set. We now describe the corresponding graphs up to isomorphism. The three families in Theorem~\ref{main_pst_dn} give rise to only two types of graphs: even cycles and lexicographic products of cycles with an independent set of size two. Thus we obtain the following graph-theoretic characterization.
	\begin{corollary}\label{pst_graphs_dn}
		Let $n\geq3$, and let $G:=\cay(D_n,S)$ be a connected Cayley graph over the dihedral group $D_n$. Then $G$ exhibits perfect state transfer if and only if it is either the even cycle $C_{2n}$ or the lexicographic product $C_n[\overline{K}_2]$ with $4$ not dividing $n$.
	\end{corollary}
	
	The remainder of the paper is organized as follows. In Section~2, we determine the eigenvalues and eigenvectors
	of Cayley graphs over dihedral groups without assuming normality (Theorems~\ref{ev1_dn} and~\ref{ev2_dn}). In Section~3, we establish necessary and sufficient conditions for exhibiting perfect state transfer on $\cay(D_n,S)$ (Theorems~\ref{pst_dn_1}, \ref{pst_dn_2}, \ref{pst_dn_3},
	and~\ref{pst_dn_4}). Finally, in Section~4, we use these conditions to obtain a complete classification of Cayley graphs over
	dihedral groups exhibiting perfect state transfer.
	
	\section{\boldmath Spectral analysis of Cayley graphs}	
	To analyze the spectrum of the Cayley graph $\mathrm{Cay}(\Gamma,S)$, we employ Fourier analysis on the finite group $\Gamma$. In particular, we use  the irreducible unitary representations of $\Gamma$ and the Fourier transform on $\Gamma$. We refer the reader to \cite{rep_symmetry, steinberg} for details.
	
	\subsection{Representations  and Fourier transform}

	Let $\Gamma$ be a finite group. A \emph{representation} of $\Gamma$ is a group homomorphism $\rho : \Gamma \to \mathrm{GL}(W),$
	where $W$ is a non-zero finite-dimensional vector space over the field $\Cl$, and $\mathrm{GL}(W)$ denotes the general linear group on $W$. The dimension of $W$ is called the \emph{degree}, denoted $d_\rho$, of the representation $\rho$. Two representations $\rho_1$ and $\rho_2$ of $\Gamma$ on vector spaces $W_1$ and $W_2$, respectively, are said to be \emph{equivalent}, denoted $\rho_1 \sim \rho_2$, if there exists an isomorphism $T : W_1 \to W_2$ such that $\rho_2(g) = T\,\rho_1(g)\,T^{-1} \quad \text{for all } g \in \Gamma.$
	
	Let $\rho : \Gamma \to \mathrm{GL}(W)$ be a representation. The \emph{character} of $\rho$ is the function $\chi_\rho : \Gamma \to \Cl$ defined by $\chi_\rho(g) = \mathrm{tr}(\rho(g))$ for all $g\in \Gamma$, where $\mathrm{tr}(\rho(g))$ denotes the trace of the matrix representation of $\rho(g)$ with respect to a basis of $W$. A subspace $X$ of $W$ is called \emph{$\Gamma$-invariant} if $\rho(g) x \in X$ for all $g \in \Gamma$ and $x \in X$. Clearly, both $\{0\}$ and $W$ are $\Gamma$-invariant subspaces, referred to as \emph{trivial} invariant subspaces. If $W$ contains no nontrivial $\Gamma$-invariant subspaces, then $\rho$ is said to be an \emph{irreducible representation} of $\Gamma$, and its character $\chi_\rho$ is called an \emph{irreducible character} of $\Gamma$. 
	
	Let $W$ be an inner product space. A representation $\rho : \Gamma \to \mathrm{GL}(W)$ is called \emph{unitary} if $\rho(g)$ is a unitary operator for every $g \in \Gamma$. It is  well known that every representation of a finite group is equivalent to a unitary representation. Let $\widehat{\Gamma}$ denote a complete set of pairwise inequivalent irreducible unitary representations of $\Gamma$.

	Let $L(\Gamma) = \Cl^\Gamma = \{f : \Gamma \to \Cl\}$. Then $L(\Gamma)$ is an inner product space, where addition and scalar multiplication are defined by $(f_1 + f_2)(g) = f_1(g) + f_2(g)$ and $(cf)(g) = c\,f(g)$
	for all $f_1, f_2, f \in L(\Gamma)$, $c \in \Cl$, and $g \in \Gamma$. The inner product on $L(\Gamma)$ is given by
	\[
	\ip{f_1}{f_2}_{L(\Gamma)} = \sum_{g \in \Gamma} f_1(g)\,\overline{f_2(g)}.
	\]
	For each $g \in \Gamma$, define a function $\delta_g \in L(\Gamma)$ by
	\[
	\delta_g(x) =\left\{ 
	\begin{array}{rl}
		1 & \text{if }x=g \\
		0& \text{otherwise.}
	\end{array}
	\right.
	\]
	It is straightforward to verify that the set $\{\delta_g : g \in \Gamma\}$ forms an orthogonal basis for $L(\Gamma)$. Hence $\dim(L(\Gamma)) = |\Gamma|$. For a subset $S \subseteq \Gamma$, define a linear map $f_S : L(\Gamma) \to L(\Gamma)$ by
	\[
	f_S\!\left(\sum_{g \in \Gamma} c_g \delta_g\right)
	= \sum_{s \in S}\sum_{g \in \Gamma} c_g \delta_{sg}.
	\]
	One can show that the matrix of $f_S$ with respect to the basis $\{\delta_g : g \in \Gamma\}$  is equal to the adjacency matrix of the Cayley graph $\cay(\Gamma, S)$. Hence the eigenvalues of $f_S$ and the eigenvalues of the adjacency matrix of $\cay(\Gamma, S)$ are the same. Considering the canonical isomorphism of $L(\Gamma)$ with $\Cl^{|\Gamma|}$, we also find that each eigenvector of $f_S$ with eigenvalue $\ld$ corresponds to an eigenvector  of the adjacency matrix of $\cay(\Gamma, S)$ with the same eigenvalue $\ld$.  Further, for $f\in L(\Gamma)$ and $g\in\Gamma$, one can find that
	\[f_S(f)(g)=\sum_{s\in S} f(s^{-1}g).\]

	Let $W_{\rho}$ denote the representation space associated with an irreducible unitary representation of $\rho \in \widehat{\Gamma}$, that is, $\rho:\Gamma\to\gl(W_\rho)$ is a group homomorphism. For each $\rho \in \widehat{\Gamma}$, let 
	$
	\{v_{1}^{\rho}, \ldots, v_{d_{\rho}}^{\rho}\}
	$
	be an orthonormal basis of $W_{\rho}$.  For $1 \leq i,j \leq d_{\rho}$, define functions $\varphi_{i,j}^{\rho} : \Gamma \to \Cl$ by
	\begin{equation}\label{vm}
		\varphi_{i,j}^{\rho}(g)= \ip{\rho(g)v_{j}^{\rho}}{v_{i}^{\rho}}_{W_{\rho}} \quad\text{for } g \in \Gamma.
	\end{equation}

	\begin{lemma}[{\cite[Corollary~1.5.8]{rep_symmetry}}]\label{ortho_rep}
		The collection
		$\{\varphi_{i,j}^{\rho} : \rho \in \widehat{\Gamma},\ 1 \leq i,j \leq d_{\rho}\}$
		forms an orthogonal basis of the space $L(\Gamma)$.
	\end{lemma}
	Define
	\[
	\mathcal{H}(\Gamma)=\bigoplus_{\sigma \in \widehat{\Gamma}} \homo(W_{\sigma}, W_{\sigma}),
	\] 
	where $\homo(W_{\sigma}, W_{\sigma})$ denotes the set of all linear maps from $W_{\sigma}$ to $W_{\sigma}$.
	
	For $\rho,\sigma \in \widehat{\Gamma}$ and $1 \le i,j \le d_{\rho}$, define $T_{i,j}^{\rho}(\sigma):W_\sigma\to W_\sigma$ by
	\[T_{i,j}^\rho(\sigma)(w)=\left\{
	\begin{array}{ll}
		\ip{w}{v_{j}^{\rho}}_{W_{\rho}} v_{i}^{\rho}& \text{ if $\sigma=\rho$}\\
		0 &\text{otherwise.}
	\end{array}
	\right.\]
	Let $T_{i,j}^\rho=\oplus_{\sigma \in \widehat{\Gamma}}T_{i,j}^\rho(\sigma)$. Note that $T_{i,j}^\rho\in\mathcal{H}(\Gamma)$ for each $\rho\in\widehat{\Gamma}$.
	\begin{lemma}[{\cite[Lemma~1.5.10]{rep_symmetry}}]
		The set
		$\{T_{i,j}^{\rho} : \rho \in \widehat{\Gamma},\ 1 \le i,j \le d_{\rho}\}$
		forms an orthogonal basis for $\mathcal{H}(\Gamma)$.
	\end{lemma}
	For each  $f \in L(\Gamma)$, the \emph{Fourier transform} of $f$ is defined by
	\[
	\mathcal{F}(f)=\bigoplus_{\rho \in \widehat{\Gamma}} \rho(f),
	\]
	where the operator $\rho(f): W_{\rho} \to W_{\rho}$ is given by
	\[
	\rho(f)(w)=\sum_{g \in \Gamma} f(g)\,\rho(g)(w) \quad\text{for } w \in W_{\rho}.
	\]
	Equivalently,
	\[
	\rho(f)=\sum_{g \in \Gamma} f(g)\,\rho(g),
	\]
	and therefore $\rho(f) \in \mathrm{Hom}(W_{\rho}, W_{\rho})$ for each $\rho \in \widehat{\Gamma}$.
	\begin{lemma}[{\cite[Theorem~1.5.11]{rep_symmetry}}]\label{fourier}
		The Fourier transform $\mathcal{F}$ is an algebra isomorphism from $L(\Gamma)$ to $\mathcal{H}(\Gamma)$. Moreover, $\mathcal{F}\compconj{\varphi_{i,j}^\rho}=\frac{|\Gamma|}{d_\rho}T_{i,j}^\rho$ and $\mathcal{F}^{-1}T_{i,j}^\rho=\frac{d_\rho}{|\Gamma|} \compconj{\varphi_{i,j}^\rho}$ for each $\rho\in\widehat{\Gamma} $ and  $1\leq i,j\leq d_\rho$.
	\end{lemma}
	Define $\mathfrak{f}\in\homo(\mathcal{H}(\Gamma),\mathcal{H}(\Gamma))$  by $\mathfrak{f}=\mathcal{F}f_S\mathcal{F}^{-1}$. Then
	$\mathfrak{f}(T_{i,j}^\rho)=\frac{d_\rho}{|\Gamma|} \mathcal{F}f_S\compconj{\varphi_{i,j}^\rho}.$ Now
	\begin{align*}
		f_{S}\overline{\varphi^{\rho}_{i,j}}(g) &= \sum_{s\in S}\overline{\varphi^{\rho}_{i,j}}(s^{-1}g)\\
		&= \sum_{s\in S}\overline{\ip{\rho(s^{-1}g)v^{\rho}_{j}}{v^{\rho}_{i}}}_{W_\rho}\\
		&= \sum_{s\in S}\overline{\ip{\rho(g)v^{\rho}_{j}}{\rho(s)v^{\rho}_{i}}}_{W_\rho}\\
		&= \sum_{s\in S}\overline{\ip{\rho(g)v^{\rho}_{j}}{\sum_{k=1}^{d_{\rho}}\ip{\rho(s)v^{\rho}_{i}}{v^{\rho}_{k}}_{W_\rho}v^{\rho}_{k}}}_{W_\rho}\\ 
		&= \sum_{s\in S}\sum_{k=1}^{d_{\rho}}\ip{\rho(s)v^{\rho}_{i}}{v^{\rho}_{k}}_{W_\rho}\overline{\ip{\rho(g)v^{\rho}_{j}}{v^{\rho}_{k}}}_{W_\rho}\\
		&= \sum_{k=1}^{d_{\rho}}\ip{\sum_{s\in S}\rho(s)v^{\rho}_{i}}{v^{\rho}_{k}}_{W_\rho}\overline{\ip{\rho(g)v^{\rho}_{j}}{v^{\rho}_{k}}}_{W_\rho}\\
		&= \sum_{k=1}^{d_{\rho}}\ip{\sum_{s\in S}\rho(s)v^{\rho}_{i}}{v^{\rho}_{k}}_{W_\rho}\overline{\varphi^{\rho}_{k,j}(g)}.
	\end{align*}
	Thus
	\[ \mathfrak{f}(T^{\rho}_{i,j})=\frac{d_{\rho}}{|\Gamma|}\sum_{k=1}^{d_{\rho}}{\ip{\sum_{s\in S}\rho(s)v^{\rho}_{i}}{v^{\rho}_{k}}}_{W_\rho}\mathcal{F}\overline{\varphi^{\rho}_{k,j}}.\]
	Therefore by Lemma~\ref{fourier}, 
	\begin{equation}\label{ev_rep}
		\mathfrak{f}(T^{\rho}_{i,j})=\sum_{k=1}^{d_{\rho}}{\ip{\sum_{s\in S}\rho(s)v^{\rho}_{i}}{v^{\rho}_{k}}}_{W_\rho}T^{\rho}_{k,j},\quad \rho\in\widehat{\Gamma} ,1\leq i,j\leq d_{\rho}.
	\end{equation}
	
	For normal Cayley graphs, the eigenvalues and eigenvectors of the adjacency matrix are well studied (see, for instance, \cite[pp. 69--70]{steinberg} and \cite[Theorem~9]{ramanujan_graphs}). In contrast, the eigenvalues and eigenvectors of the adjacency matrix of non-normal Cayley graphs have received considerably less attention in the literature. In the next section, we present the eigenvalues and eigenvectors of the adjacency matrix of $\cay(D_n, S)$ without assuming that $S$ is normal.  The non-normal case was first considered by Xiwang Cao, Bocong Chen, and San Ling in the context of continuous-time quantum walks in the paper ``Perfect state transfer on Cayley graphs over dihedral groups: the non-normal case, \emph{Electron.\ J.\ Combin.}, 27(2):Paper No.\ 2.28, 18, 2020". However, the paper was later retracted due to errors in the main results.

	\subsection{\boldmath Eigenvalues and eigenvectors of $\cay(D_n,S)$}
	The irreducible representations of $D_n$ are given in the following  result. We denote by $\omega_n$ the complex number $e^{\frac{2\pi \iu}{n}}$.
	\begin{lemma}[{\cite[pp.~37--38]{rep_dihedral}}]
		Let $n$ be a positive integer. Then the irreducible representations of the dihedral group $D_n$ are given in Table~\ref{irr_n_odd} for odd $n$ and in Table~\ref{irr_n_even} for even $n$.
	\end{lemma}		
	\begin{table}[h]
		\centering
		\begin{tabular}{|c|c|c|}
			\hline
			& $a^{k}\ (0 \le k \le n-1)$ & $ba^{k}\ (0 \le k \le n-1)$ \\
			\hline
			$\psi_{1}$ & $1$ & $1$ \\
			$\psi_{2}$ & $1$ & $-1$ \\
			$\rho_h \ \left(1 \le h \le \lfloor (n-1)/2 \rfloor\right)$
			&
			$\begin{pmatrix}
				\omega_n^{hk} & 0 \\
				0 & \omega_n^{-hk}
			\end{pmatrix}$
			&
			$\begin{pmatrix}
				0 & \omega_n^{-hk} \\
				\omega_n^{hk} & 0
			\end{pmatrix}$ \\
			\hline
		\end{tabular}
		\caption{Irreducible representations of $D_n$ ($n$ odd).}
		\label{irr_n_odd}
	\end{table}
	\begin{table}[h]
		\centering
		\begin{tabular}{|c|c|c|}
			\hline
			& $a^{k}\ (0 \le k \le n-1)$ & $ba^{k}\ (0 \le k \le n-1)$ \\
			\hline
			$\psi_{1}$ & $1$ & $1$ \\
			$\psi_{2}$ & $1$ & $-1$ \\
			$\psi_{3}$ & $(-1)^k$ & $(-1)^k$ \\
			$\psi_{4}$ & $(-1)^k$ & $(-1)^{k+1}$ \\
			$\rho_h \ \left(1 \le h \le \lfloor (n-1)/2 \rfloor\right)$
			&
			$\begin{pmatrix}
				\omega_n^{hk} & 0 \\
				0 & \omega_n^{-hk}
			\end{pmatrix}$
			&
			$\begin{pmatrix}
				0 & \omega_n^{-hk} \\
				\omega_n^{hk} & 0
			\end{pmatrix}$ \\
			\hline
		\end{tabular}
		\caption{Irreducible representations of $D_n$ ($n$ even).}
		\label{irr_n_even}
	\end{table}	
	Let $\cay(D_n,S)$ be a Cayley graph with $1\notin S$ and $S=S^{-1}$. 
	Define $S_a = S \cap \langle a\rangle$ and $S_b = bS \cap \langle a\rangle$.
	We further decompose $S_a$ and $S_b$ according to the parity of the exponent of $a$.
	\begin{align*}
		S_{a,0}&=\left\{a^k: 1\leq k \leq n-1,~a^k\in S_a,~k\equiv0\pmod 2\right\},\\
		S_{a,1}&=\left\{a^k: 1\leq k \leq n-1,~a^k\in S_a,~k\equiv1\pmod 2\right\},\\
		S_{b,0}&=\left\{a^k: 0\leq k \leq n-1,~a^k\in S_b,~k\equiv0\pmod 2\right\}, \text{ and }\\
		S_{b,1}&=\left\{a^k: 0\leq k \leq n-1,~a^k\in S_b,~k\equiv1\pmod 2\right\}.
	\end{align*}	
	Throughout the paper, we define
	\begin{equation}\label{crb2216}
		d_a = |S_a|,\quad d_b = |S_b|, \quad d_{a,i} = |S_{a,i}|,\quad d_{b,i} = |S_{b,i}| 
		\quad\text{for } i \in \{0,1\}~\text{and $d = |S|$.}
	\end{equation}
	For any complex number $z$,  we write $\{z^{k}\}_{k=0}^{n-1} := (z^{0},\,z^{1},\,\ldots,\,z^{(n-1)})$ to denote a row vector. Similarly, for two complex numbers $z$ and $w$, $\left(\{z^{k}\}_{k=0}^{n-1}, \{w^{k}\}_{k=0}^{r-1}\right):=\left(z^{0},\,z^{1},\,\ldots,\,z^{(n-1)},w^{0},\,w^{1},\,\ldots,\,w^{(r-1)} \right)$. Furthermore, $\mathbf{0}_n$ and $\mathbf{1}_n$ represent 
	the row vectors consisting of $n$ zeros and $n$ ones, respectively, and $-\mathbf{1}_n$ denotes the row 
	vector of $n$ entries, each equal to $-1$.

	We now focus on the eigenvalues and eigenvectors of the adjacency matrix of $\cay(D_n,S)$ for even $n$. Since the case for odd $n$ can be treated as a special instance of the even case (see Tables~\ref{irr_n_odd} and \ref{irr_n_even}), we restrict our attention to even $n$. The following theorem provides the eigenvalues and the corresponding normalized eigenvectors of $\cay(D_n,S)$ associated with the one-dimensional representations of $D_n$ for even $n$. 
	\begin{theorem}\label{ev1_dn}
		For even $n$, the eigenvalues and the corresponding normalized eigenvectors of the adjacency matrix of 
		$\cay(D_n, S)$ associated with the one-dimensional representations of $D_n$ are as follows.
		\begin{enumerate}
			\item[(i)] For the character $\psi_1$, the eigenvalue and a corresponding eigenvector are respectively given by
			\begin{equation*}
				\widetilde{\lambda}_1 = d\quad\text{and}\quad
				v_1 = \frac{1}{\sqrt{2n}}
				\mathbf{1}_{2n}^{t}.
			\end{equation*}
			
			\item[(ii)] For the character $\psi_2$, the eigenvalue and a corresponding eigenvector are respectively given by
			\begin{equation*}
				\widetilde{\lambda}_2  = d_a - d_b\quad\text{and}\quad
				v_2 = \frac{1}{\sqrt{2n}}
				\begin{pmatrix}
					\mathbf{1}_{n},
					-\mathbf{1}_{n}
				\end{pmatrix}^{t}.
			\end{equation*}
			\item[(iii)] For the character $\psi_3$, the eigenvalue and a corresponding eigenvector are respectively given by
			\begin{align*}
				\widetilde{\lambda}_3  &= d_{a,0}-d_{a,1}+d_{b,0}-d_{b,1}\quad\text{and}\\
				v_3 &= \frac{1}{\sqrt{2n}}
				\begin{pmatrix}
					\{(-1)^i\}_{i=0}^{ n-1},\;
					\{(-1)^i\}_{i=0}^{ n-1}
				\end{pmatrix}^{t}.
			\end{align*}
			\item[(iv)] For the character $\psi_4$, the eigenvalue and a corresponding eigenvector are respectively given by
			\begin{align*}
				\widetilde{\lambda}_4 &= d_{a,0}-d_{a,1}+d_{b,1}-d_{b,0}\quad\text{and}\\
				v_4 &= \frac{1}{\sqrt{2n}}
				\begin{pmatrix}
					\{(-1)^i\}_{i=0}^{ n-1},\;
					\{(-1)^{i+1}\}_{i=0}^{ n-1}
				\end{pmatrix}^{t}.
			\end{align*}
		\end{enumerate}
	\end{theorem}
	\begin{proof}
		\begin{enumerate}
			\item[(i)] For the character $\psi_1$, we have $W_{\psi_1}=\Cl$. Therefore $d_{\psi_{1}}=1$ and $\{v_1^{\psi_1}\}=\{1\}$. Hence by \eqref{ev_rep}, we obtain
			\[\mathfrak{f}(T^{\psi_{1}}_{1,1})=\ip{\sum_{s\in S}\psi_1(s)v^{\psi_{1}}_{1}}{v^{\psi_{1}}_{1}}T^{\psi_{1}}_{1,1}=\sum_{s\in S} \psi_1(s)T^{\psi_1}_{1,1}=dT^{\psi_{1}}_{1,1}.\]
			Consequently,
			\[f_{S}\mathcal{F}^{-1}T^{\psi_{1}}_{1,1}=d\mathcal{F}^{-1}T^{\psi_{1}}_{1,1}.\]
			
			Let $g\in\Gamma$. From Lemma~\ref{fourier}, we have 
			$$\mathcal{F}^{-1}T^{\psi_{1}}_{1,1}(g)=\frac{1}{2n} \overline{\varphi_{1,1}^{\psi_1}}(g)=\frac{1}{2n}\overline{\psi_1(g)}=\frac{1}{2n}.$$			
			Thus $\widetilde{\lambda}_1 = d$ and the corresponding normalized eigenvector is $v_1 = \frac{1}{\sqrt{2n}}
			\mathbf{1}_{2n}^{t}$.
			\item[(ii)] Similar to Part (i), we find that $\widetilde{\lambda}_2  = d_a - d_b$ and 
			$v_2 = \frac{1}{\sqrt{2n}}
			\begin{pmatrix}
				\mathbf{1}_{n},
				-\mathbf{1}_{n}
			\end{pmatrix}^{t}.$
			\item[(iii)] For the character $\psi_3$ also, $d_{\psi_{3}}=1$ and  $\{v_1^{\psi_3}\}=\{1\}$. Applying \eqref{ev_rep},  we have
			\[\mathfrak{f}(T^{\psi_{3}}_{1,1})=\ip{\sum_{s\in S}\psi_3(s)v^{\psi_{3}}_{1}}{v^{\psi_{3}}_{1}}T^{\psi_{3}}_{1,1}=(d_{a,0}-d_{a,1}+d_{b,0}-d_{b,1})T^{\psi_{3}}_{1,1}.\]
			This implies that
			\[f_{S}\mathcal{F}^{-1}T^{\psi_{3}}_{1,1}=(d_{a,0}-d_{a,1}+d_{b,0}-d_{b,1})\mathcal{F}^{-1}T^{\psi_{3}}_{1,1},\]
			yielding $\widetilde{\lambda}_3 = d_{a,0} - d_{a,1} + d_{b,0} - d_{b,1}$ and  the corresponding normalized eigenvector is
			\[
			v_3 = \frac{1}{\sqrt{2n}}
			\begin{pmatrix}
				\{(-1)^i\}_{i=0}^{ n-1},\;
				\{(-1)^i\}_{i=0}^{ n-1}
			\end{pmatrix}^{t}.
			\]
			\item[(iv)] This part is similar to Part (iii). \qedhere
		\end{enumerate}
	\end{proof}
	For $C\subset \langle a\rangle \subset D_n$, define $C^* := \{k \in \mathbb{Z}_n : a^k \in C\}$. Then for $h\in\Zl_n$, define
	\[
	\eta_h(C) := \sum_{k \in C^*} \omega_n^{hk},
	\]
	with the convention that, if $C=\emptyset$ then $\eta_h(C)=0$. We now study the eigenvalues and eigenvectors corresponding to the two-dimensional representations $\rho_h$ of $D_n$.
	\begin{theorem}\label{ev2_dn}
		The eigenvalues of the adjacency matrix of $\cay(D_n,S)$ corresponding to $\rho_h$ are given by
		\[
		\ld_h^{(1)}=\eta_h(S_a)+|\eta_h(S_b)|\quad\text{and}\quad
		\ld_h^{(2)}=\eta_h(S_a)-|\eta_h(S_b)|,
		\]
		each of multiplicity $2$.  
		Moreover, the corresponding eigenvectors are given as follows.
		\begin{enumerate}
			\item[(i)] If $\eta_h(S_b)=0$, then the normalized eigenvectors corresponding to $\ld_h:=\eta_h(S_a)$ are
			\begin{align*}
				u^{(1)}_h&= \frac{1}{\sqrt{n}}	\begin{pmatrix}	\{\omega_n^{hk}\}_{k=0}^{n-1},\mathbf{0}_{n}\end{pmatrix}^{t},\\
				u^{(2)}_h&=\frac{1}{\sqrt{n}}\begin{pmatrix}\mathbf{0}_{n},\{\omega_n^{hk}\}_{k=0}^{n-1}	\end{pmatrix}^{t},\\
				u^{(3)}_h&=  \frac{1}{\sqrt{n}}\begin{pmatrix}\mathbf{0}_{n},\{\omega_n^{-hk}\}_{k=0}^{n-1}	\end{pmatrix}^{t},\text{ and}\\
				u^{(4)}_h&=\frac{1}{\sqrt{n}}\begin{pmatrix}	\{\omega_n^{-hk}\}_{k=0}^{n-1},\mathbf{0}_{n}	\end{pmatrix}^{t}.
			\end{align*}
			\item[(ii)] If $\eta_h(S_b)\neq 0$, let $\ell_h=\frac{\overline{\eta_h(S_b)}}{|\eta_h(S_b)|}$.
			Then the normalized eigenvectors corresponding to $\ld_h^{(1)}$ are
			\begin{align*}
				v^{(1)}_h&=	\frac{1}{\sqrt{2n}}\left(\bar{\ell}_{h},\bar{\ell}_{h}{\omega}_n^{h}, \cdots, \bar{\ell}_{h}{\omega}_n^{(n-1)h},1,{\omega}_n^{h}, \cdots,{\omega}_n^{(n-1)h} \right)^t\text{ and}\\
				v^{(2)}_h&=\frac{1}{\sqrt{2n}}\left(1,{\omega}_n^{-h}, \cdots,{\omega}_n^{-(n-1)h}, \bar{\ell}_{h},\bar{\ell}_{h}{\omega}_n^{-h}, \cdots, \bar{\ell}_{h}{\omega}_n^{-(n-1)h}\right)^t,
			\end{align*}
			and the normalized eigenvectors corresponding to  $\ld_h^{(2)}$ are
			\begin{align*}
				v^{(3)}_h&= \frac{1}{\sqrt{2n}}\left(-1,-{\omega}_n^{h}, \cdots,-{\omega}_n^{(n-1)h}, {\ell}_{h},{\ell}_{h} {\omega}_n^{h}, \cdots, {\ell}_{h} {\omega}_n^{(n-1)h}\right)^t\text{ and}\\
				v^{(4)}_h&= \frac{1}{\sqrt{2n}}\left({\ell}_{h},{\ell}_{h} {\omega}_n^{-h}, \cdots, {\ell}_{h} {\omega}_n^{-(n-1)h},-1,-{\omega}_n^{-h}, \cdots,-{\omega}_n^{-(n-1)h} \right)^t.
			\end{align*}
		\end{enumerate}
	\end{theorem}
	\begin{proof}
		For $1\leq h\leq\left\lfloor{(n-1)}/{2}\right\rfloor$,  recall from the Table~\ref{irr_n_even} that $\rho_{h}(a^{k})v_{1}^{\rho_{h}}= \omega_n^{hk}v_{1}^{\rho_{h}},$ $\rho_{h}(a^{k})v_{2}^{\rho_{h}}=\omega_n^{-hk}v_{2}^{\rho_{h}},$ $
		\rho_{h}(ba^k)v_{1}^{\rho_{h}}=\omega_n^{hk}v_{2}^{\rho_{h}},$ and $
		\rho_{h}(ba^k)v_{2}^{\rho_{h}}=\omega_n^{-hk}v_{1}^{\rho_{h}}.$
		Therefore by \eqref{ev_rep}, we find
		\begin{align*}
			\mathfrak f(T_{1,1}^{\rho_h})&=\eta_h(S_a)T_{1,1}^{\rho_h}+\eta_h(S_b)T_{2,1}^{\rho_h},\\
			\mathfrak f(T_{1,2}^{\rho_h})&=\eta_h(S_a)T_{1,2}^{\rho_h}+\eta_h(S_b)T_{2,2}^{\rho_h},\\
			\mathfrak f(T_{2,1}^{\rho_h})&=\eta_h(S_b^{-1})T_{1,1}^{\rho_h}+\eta_h(S_a)T_{2,1}^{\rho_h},\text{ and}\\
			\mathfrak f(T_{2,2}^{\rho_h})&=\eta_h(S_b^{-1})T_{1,2}^{\rho_h}+\eta_h(S_a)T_{2,2}^{\rho_h}.
		\end{align*}
		Thus the matrix of $\mathfrak{f}$ restricted to the subspace $\mathrm{span}\{T_{1,1}^{\rho_h},T_{2,1}^{\rho_h},T_{1,2}^{\rho_h},T_{2,2}^{\rho_h}\}$ with respect to the orthogonal basis $\{T_{1,1}^{\rho_h},T_{2,1}^{\rho_h},T_{1,2}^{\rho_h},T_{2,2}^{\rho_h}\}$  is
		\begin{equation}\label{mat}
			\begin{pmatrix}
				\eta_h(S_a)&\eta_h(S_b^{-1})&0&0\\
				\eta_h(S_b)&\eta_h(S_a)&0&0\\
				0&0&\eta_h(S_a)&\eta_h(S_b^{-1})\\
				0&0&\eta_h(S_b)&\eta_h(S_a)
			\end{pmatrix}.
		\end{equation}
		The matrix in \eqref{mat} is block‑diagonal with two identical $2\times2$ diagonal blocks
		\[
		M_h=\begin{pmatrix}
			\eta_h(S_a)&\eta_h(S_b^{-1})\\
			\eta_h(S_b)&\eta_h(S_a)
		\end{pmatrix}.
		\]
		The eigenvalues of $M_h$ are $\eta_h(S_a)\pm|\eta_h(S_b)|$. Thus the eigenvalues of $\cay(D_n,S)$ corresponding to $\rho_h$ are $\ld_h^{(1)}$ and $\ld_h^{(2)}$, each with multiplicity $2$.

		\noindent\textbf{Case 1.}   $\eta_h(S_b)=0$. In this case, $M_h=\eta_h(S_a)I_2$, and therefore every nonzero vector is an eigenvector of $M_h$.
		In particular, $T_{1,1}^{\rho_h}$, $T_{2,1}^{\rho_h}$, $T_{1,2}^{\rho_h}$, $T_{2,2}^{\rho_h}$ are eigenvectors of $\mathfrak{f}$ with corresponding eigenvalue $\eta_h(S_a)$. Consequently,
		$\varphi_{1,1}^{\rho_h},\varphi_{2,1}^{\rho_h},
		\varphi_{1,2}^{\rho_h},\varphi_{2,2}^{\rho_h}$ are eigenvectors of $f_S$
		with corresponding eigenvalue $\eta_h(S_a)$. Now for $g\in D_n$, calculating $\varphi_{1,1}^{\rho_h}(g),\varphi_{2,1}^{\rho_h}(g),
		\varphi_{1,2}^{\rho_h}(g)$ and $\varphi_{2,2}^{\rho_h}(g)$ using \eqref{vm}, we find the normalized eigenvectors $u_h^{(1)}, u_h^{(2)}, u_h^{(3)}$ and $u_h^{(4)}$, respectively.		
		
		\noindent\textbf{Case 2.} $\eta_h(S_b)\neq 0$.
		In this case, $M_h$ has two distinct eigenvalues $\ld_h^{(1)}$ and $\ld_h^{(2)}$.
		For the eigenvalue $\ld_h^{(1)}$, we have
		\[
		M_h
		\begin{pmatrix}
			\ell_h\\
			1
		\end{pmatrix}
		=
		\ld_h^{(1)}
		\begin{pmatrix}
			\ell_h\\
			1
		\end{pmatrix},
		\]
		and hence $(\ell_h,1)^{t}$ is an eigenvector of $M_h$ with corresponding eigenvalue $\ld_h^{(1)}$.
		Therefore\linebreak[4] $\ell_h\,T_{1,1}^{\rho_h}+T_{2,1}^{\rho_h}\text{ and }		\ell_h\,T_{1,2}^{\rho_h}+T_{2,2}^{\rho_h}$
		are eigenvectors of $\mathfrak{f}$ with corresponding eigenvalue $\ld_h^{(1)}$.
		Consequently, $
		\bar\ell_h\,\varphi_{1,1}^{\rho_h}+\varphi_{2,1}^{\rho_h}\text{ and }
		\bar\ell_h\,\varphi_{1,2}^{\rho_h}+\varphi_{2,2}^{\rho_h}$
		are eigenvectors of $f_S$ with corresponding eigenvalue $\ld_h^{(1)}$.
		Now for $g\in D_n$, calculating $
		(\bar\ell_h\,\varphi_{1,1}^{\rho_h}+\varphi_{2,1}^{\rho_h})(g)$ and $(\bar\ell_h\,\varphi_{1,2}^{\rho_h}+\varphi_{2,2}^{\rho_h})(g)$  using \eqref{vm}, we obtain the normalized eigenvectors
		$v_h^{(1)}$ and $v_h^{(2)}$, respectively.
		
		For the eigenvalue $\ld_h^{(2)}$ of $M_h$, a corresponding eigenvector is
		$(-1,\bar\ell_h)^{T}$.	Proceeding as in the previous case, we find the normalized eigenvectors
		$v_h^{(3)}$ and $v_h^{(4)}$ of $\cay(D_n,S)$ with corresponding eigenvalue $\ld_h^{(2)}$.
	\end{proof}

	Recall that $\dim L(D_n)=|D_n|=2n$, and 
	$\bigl\{\varphi_{i,j}^\rho : \rho\in\widehat{D}_n,\,1\le i,j\le d_\rho\bigr\}$
	forms an orthogonal basis of $L(D_n)$. For each irreducible
	representation $\rho$ of $D_n$, let 
	$V_\rho := \operatorname{span}\{\varphi_{i,j}^\rho : 1\le i,j\le d_\rho\}.$
	Then
	\[
	L(D_n) \;=\; \bigoplus_{\rho\in\widehat{D}_n} V_\rho,
	\]
	and each $V_\rho$ is invariant under  $f_S$ for any $S\subseteq D_n$. Note that the eigenvalues and eigenvectors described in Theorems~\ref{ev1_dn} and \ref{ev2_dn} correspond to that of $f_S\in L(D_n)$. Thus if $n$ is odd, then
	\[\{v_1,v_2\}\ \cup\ \bigcup_{\substack{h=1\\ \eta_h(S_b)=0}}^{\left\lfloor\frac{n-1}{2}\right\rfloor}
	\left\{u_h^{(1)},u_h^{(2)},u_h^{(3)},u_h^{(4)}\right\}\ \cup\ \bigcup_{\substack{h=1\\ \eta_h(S_b)\neq 0}}^{\left\lfloor\frac{n-1}{2}\right\rfloor}
	\left\{v_h^{(1)},v_h^{(2)},v_h^{(3)},v_h^{(4)}\right\}\]
	is an orthonormal basis of $\Cl^{2n}$. Similarly, if $n$ is even,  then 
	\[\{v_1,v_2,v_3,v_4\}\ \cup\ \bigcup_{\substack{h=1\\ \eta_h(S_b)=0}}^{\left\lfloor\frac{n-1}{2}\right\rfloor}
	\left\{u_h^{(1)},u_h^{(2)},u_h^{(3)},u_h^{(4)}\right\}\ \cup\ \bigcup_{\substack{h=1\\ \eta_h(S_b)\neq 0}}^{\left\lfloor\frac{n-1}{2}\right\rfloor}
	\left\{v_h^{(1)},v_h^{(2)},v_h^{(3)},v_h^{(4)}\right\}\]
	is an orthonormal basis of $\Cl^{2n}$.

	\section{Characterization for perfect state transfer}
	Let $\Gamma$ be a finite group. For each $g \in \Gamma$, define two permutations of $\Gamma$, namely the left multiplication $\mathcal{L}_g$ and the right multiplication $\mathcal{R}_g$, by $\mathcal{L}_g(h) = gh$ and $\mathcal{R}_g(h) = hg$, respectively, for all $h \in \Gamma$. For a permutation $\sigma$ of $\Gamma$, the corresponding permutation matrix in
	$\mathbb{C}^{\Gamma \times \Gamma}$ has $uv$-th entry $ \delta_{u}(\sigma(v)).$
	In what follows, we use the same symbol $\sigma$ to denote both the permutation
	and its associated permutation matrix, depending on the context. Let $Z(\Gamma):=\{g\in\Gamma: gh=hg~\text{for each } h\in\Gamma\}$ be the center of the group $\Gamma$. For a graph $G$, let $\aut(G)$ denote the group of
	all automorphisms of $G$.
	\begin{lemma}\label{pst_cayley}
		Let $G:=\cay(\Gamma,S)$ be a Cayley graph over a finite group $\Gamma$ with connection set $S$. Then $G$ exhibits perfect state transfer at time $\tau$ if and only if $T_\tau(P)$ is a permutation matrix of order $2$ with no fixed points. Moreover, if $G$ is normal, then $T_\tau(P)=\mathcal{R}_z$ for some $z\in Z(\Gamma)$ of order $2$.
	\end{lemma}
	\begin{proof}
		Suppose $\cay(\Gamma,S)$ exhibits perfect state transfer between the vertices $u$ and $v$ at time $\tau$. By Lemma~\ref{defpst},  we have $ T_{\tau}(P)\eu=\ev$.  Using Lemma~\ref{reg}, $T_\tau(P)$ is a polynomial in $A$. Further, $\mathbf{e}^t_{xg}A^k\mathbf{e}_{yg} =\mathbf{e}^t_{x}A^k\mathbf{e}_{y}$ for any $x,y,g\in\Gamma$ and nonnegative integer $k$. Thus it follows that $\mathbf{e}^t_{xg}T_\tau(P)\mathbf{e}_{yg} =\mathbf{e}^t_{x}T_\tau(P)\mathbf{e}_{y}$. This,  together with the symmetricity of $T_\tau(P)$ and $T_{\tau}(P)\eu=\ev$, implies that $T_\tau(P)$ is a permutation matrix of order $2$ with no fixed points. The converse direction is immediate by Lemma~\ref{defpst}.
		
		Now assume that $G$ is a normal Cayley graph and regard $T_\tau(P)$ as a permutation of $\Gamma$. Since $G$ is normal, $\{\mathcal{L}_g,\mathcal{R}_g:g\in\Gamma\}\subseteq\aut(G)$ . 
		Further, $A\sigma = \sigma A$ for every $\sigma \in \aut(G)$ and $T_\tau(P)$ is a polynomial in $A$.  Therefore $T_\tau(P) \in Z(\aut(G)).$
		Hence $T_\tau(P)\mathcal{L}_g = \mathcal{L}_g T_\tau(P)$ for all  $g\in \Gamma.$
		Thus for all $g,h \in \Gamma$, $T_\tau(P)(gh) = gT_\tau(P)(h).$ Let $z := T_\tau(P)(1)$.  
		Setting $h=1$ in the previous identity yields $T_\tau(P)(g) = gz$ for all $g \in \Gamma$. 
		Therefore, $T_\tau(P) = \mathcal{R}_z.$ Since $T_\tau(P) \in Z(\aut(G))$, we must have $\mathcal{R}_z \mathcal{R}_g = \mathcal{R}_g \mathcal{R}_z$ for all $g \in \Gamma,$
		which holds precisely when $z \in Z(\Gamma)$.  
		Finally, since $T_\tau(P)$ is a permutation matrix of order $2$ with no fixed points, $z$ must be an element of order~$2$. 
	\end{proof}
	The following result provides a relation  between the occurrence of perfect state transfer and the periodicity of a graph.
	\begin{corollary}\label{minimum}
		Let $G$ be a connected Cayley graph. If $G$ exhibits perfect state transfer at time $\tau$, then $T_{\tau}(\mu)\in\{-1,1\}$ for every eigenvalue $\mu$ of $P$. Moreover, if $\tau$ is the minimum time for perfect state transfer, then $G$ is $2\tau$-periodic.
	\end{corollary}
	\begin{proof}
		Let $G$ exhibit perfect state transfer at time $\tau$ and let $\mu$ be an eigenvalue of $P$. Then by Lemma~\ref{pst_cayley}, $T_\tau(\mu)^2=1$, and so $T_{\tau}(\mu)\in\{-1,1\}$.
		
		Suppose $\tau$ is the minimum time for exhibiting perfect state transfer on $G$. Using properties of the Chebyshev polynomial of the first kind, it follows that $\mu=\cos \frac{m}{\tau}\pi$ for some positive integer $m$. Consequently,  Theorem~\ref{ev_grover} gives that $U^{2\tau}=I$.  Thus $G$ is periodic. Suppose that $G$ is $k$-periodic, that is, $k$ is the least positive integer such that $U^k=I$. This implies that $k\mid2\tau$. Since $\tau$ is the minimum time at which perfect state transfer occurs in $G$, we have $\tau\in\{1,\ldots,k-1\}$. Hence $k=2\tau$.
	\end{proof}
	\begin{corollary}\label{pair_minimum}
		Let $G$ be a connected Cayley graph. If $G$ exhibits perfect state transfer from $u$ to $v$ and $u$ to $w$, then $v=w$.
	\end{corollary}
	\begin{proof}
		Suppose $G$ exhibits perfect state transfer from $u$ to $v$ and from $u$ to $w$ at the minimum times $\tau$ and $\sigma$, respectively. By Corollary~\ref{minimum}, the period of $G$ is both $2\tau$ and $2\sigma$. Hence $\tau=\sigma$.   Then by Lemma~\ref{defpst}, $\ev=T_{\tau}(P)\eu=\mathbf{e}_w$. Thus $v=w$.
	\end{proof}
	Depending on the parity of $n$ and whether $S$ is normal or non-normal, we analyze perfect state transfer on $\cay(D_n,S)$ in the following four cases.
	\begin{enumerate}
		\item[(i)] $n$ is even and $S$ is non-normal (Theorem~\ref{pst_dn_1}),
		\item[(ii)] $n$ is even and $S$ is normal (Theorem~\ref{pst_dn_2}),
		\item[(iii)] $n$ is odd and $S$ is non-normal (Theorem~\ref{pst_dn_3}), and
		\item[(iv)] $n$ is odd and $S$ is normal (Theorem~\ref{pst_dn_4}).
	\end{enumerate}
	First assume that $n$ is even and write $n=2m$ for some positive integer $m$. By Lemma~\ref{reg},  the discriminant $P$ of $\cay(D_n,S)$ is given by
	$ P=\frac{1}{d}A$, where $d=|S|$.  
	For the one-dimensional representations of $D_n$, the eigenvalues of $P$ are given by
	\begin{equation}\label{ev_p_dn1}
		\widetilde{\mu}_i:=\frac{\widetilde{\lambda}_i}{d}\quad \text{for}\quad i\in\{1,2,3,4\},
	\end{equation} 
	with corresponding orthonormal eigenvectors $v_i$ given in Theorem~\ref{ev1_dn}.  
	For each two-dimensional representation of $D_n$, if $\eta_h(S_b)=0$, then the eigenvalue 
	\begin{equation}\label{ev_p_dn2}
		\mu_h:=\frac{\lambda_h}{d}
	\end{equation} 
	of $P$ has multiplicity $4$ with orthonormal eigenvectors 
	$u_h^{(1)},~u_h^{(2)},~u_h^{(3)}$ and $u_h^{(4)}$ given in Theorem~\ref{ev2_dn}.  
	If $\eta_h(S_b)\neq 0$, then $P$ has two eigenvalues 
	\begin{equation}\label{ev_p_dn3}
		\mu_h^{(1)}:=\frac{\lambda_h^{(1)}}{d} \quad\text{and}\quad \mu_h^{(2)}:=\frac{\lambda_h^{(2)}}{d},
	\end{equation}
	each having multiplicity $2$. The eigenvectors 
	$v_h^{(1)}$ and $v_h^{(2)}$ correspond to $\mu_h^{(1)}$; and 
	$v_h^{(3)}$ and $v_h^{(4)}$ correspond to $\mu_h^{(2)}$. 
	
	\begin{remark}
		Throughout the paper, for $1\le h\le \left\lfloor {(n-1)}/{2}\right\rfloor$, $\mu_h$ is understood to be defined only when $\eta_h(S_b)=0$, and $\mu_h^{(1)}$, $\mu_h^{(2)}$ only when $\eta_h(S_b)\neq 0$. 
	\end{remark}
	
	The spectral decomposition of the discriminant matrix $P$ of $\cay(D_n,S)$ is
	\begin{equation}\label{spec_P}
		\begin{aligned}
			P
			&= \sum_{i=1}^{4} \widetilde{\mu}_i\, v_i v_i^{*} 
			+ \sum_{\substack{h=1\\ \eta_h(S_b)=0}} ^{m-1}
			\mu_h \sum_{j=1}^{4} u_h^{(j)} \bigl(u_h^{(j)}\bigr)^{*} \\
			&\quad+ \sum_{\substack{h=1\\ \eta_h(S_b)\neq 0}}^{m-1}
			\biggl(
			\mu_h^{(1)} \sum_{j=1}^{2} v_h^{(j)} \bigl(v_h^{(j)}\bigr)^{*}
			\;+\;
			\mu_h^{(2)} \sum_{j=3}^{4} v_h^{(j)} \bigl(v_h^{(j)}\bigr)^{*}
			\biggr).
		\end{aligned}
	\end{equation}
	Define  $E_i = v_i v_i^{\ast}$ for $i\in\{1,2,3,4\}$. Also, for $1 \leq j \leq 4$ and $1 \leq h \leq m-1$, set
	\[
	E_h^{(j)} := u_h^{(j)}\bigl(u_h^{(j)}\bigr)^{\ast} ~~\text{if } \eta_h(S_b)=0~~\text{and}~~
	F_h^{(j)} := v_h^{(j)}\bigl(v_h^{(j)}\bigr)^{\ast} ~~ \text{if } \eta_h(S_b)\neq 0.
	\]
	Then all these matrices are idempotent and mutually orthogonal. Therefore by \eqref{spec_P}, for any positive integer $k$, we have 
	\begin{equation}\label{sd_tp_dn}
		\begin{aligned}
			T_k(P) &= \sum_{i=1}^{4} T_k(\widetilde{\mu}_i) E_i 
			+ \sum_{\substack{h=1\\ \eta_h(S_b)=0}} ^{m-1} T_k(\mu_h) \sum_{j=1}^{4} E_h^{(j)}\\
			&\quad+ \sum_{\substack{h=1\\ \eta_h(S_b)\neq 0}} ^{m-1} \left( 
			T_k(\mu_h^{(1)}) (F_h^{(1)}+F_h^{(2)}) 
			+ T_k(\mu_h^{(2)}) (F_h^{(3)}+F_h^{(4)}) 
			\right).
		\end{aligned}
	\end{equation}
	Now we compute the matrices $E_i,~E_h^{(j)} $
	and $F_h^{(j)}$. We find that 
	\begin{equation}\label{esp_1}
		\begin{aligned}
			E_1 &= \frac{1}{2n}J_{2n}, 
			&\qquad 
			E_2 &= \frac{1}{2n}\begin{pmatrix} J_n & -J_n \\ -J_n & J_n \end{pmatrix}, \\
			E_3 &= \frac{1}{2n}\Big( \zeta_1(u,v) \Big)_{u,v=0}^{2n-1}, 
			&\qquad 
			E_4 &= \frac{1}{2n}\Big( \zeta_2(u,v) \Big)_{u,v=0}^{2n-1},
		\end{aligned}
	\end{equation}
	where 
	\[\zeta_1(u,v)=\left(-1\right)^{u+v}, \qquad 0\leq u,v\leq 2n-1,\]
	and
	\[\zeta_2(u,v)=\begin{cases}
		\left(-1\right)^{u+v}, & 0\leq u,v\leq n-1\text{ or } n\leq u,v\leq 2n-1,\\
		\left(-1\right)^{u+v+1}, & \text{otherwise}.
	\end{cases}\]
	For $1\leq h\leq m-1$ with  $\eta_h(S_b)=0$ and $1\leq j\leq 4$, the matrices $E_h^{(j)}$ are given by
	\begin{equation}\label{esp_2}
		\begin{aligned}
			E_h^{(1)} &= \frac{1}{n}\begin{pmatrix}{\Omega}_h & 0\\ 0 & 0 \end{pmatrix}, 
			&\qquad 
			E_h^{(2)} &= \frac{1}{n}\begin{pmatrix} 0 & 0\\ 0 &  {\Omega}_h \end{pmatrix}, \\
			E_h^{(3)} &= \frac{1}{n}\begin{pmatrix} 0 & 0\\ 0 & \overline{\Omega}_h \end{pmatrix}\text{ and} 
			&\qquad 
			E_h^{(4)} &= \frac{1}{n}\begin{pmatrix} \overline{{\Omega}}_h & 0\\ 0 & 0 \end{pmatrix},
		\end{aligned}
	\end{equation}
	where $(\Omega_h)_{u,v}=\omega_n^{(u-v)h}$ for $0\leq u,v\leq n-1$, that is,
	\[
	\Omega_h=
	\begin{pmatrix}
		1 &\omega_n^{-h}&\cdots & \omega_n^{-(n-1)h}\\
		\omega_n^{h} & 1 &\cdots & \omega_n^{-(n-2)h}\\
		\vdots&\vdots&\ddots&\vdots\\
		\omega_n^{(n-1)h} &\omega_n^{(n-2)h}&\cdots & 1
	\end{pmatrix}.
	\]
	For $1\leq h\leq m-1$ with  $\eta_h(S_b)\neq0$ and $1\leq j\leq 4$, the matrices $F_h^{(j)}$ are given by
	\begin{equation}\label{esp_3}
		\begin{aligned}
			F_h^{(1)} &= \frac{1}{2n}\begin{pmatrix}{\Omega}_h & \bar{\ell}_h\Omega_h \\ \ell_h {\Omega}_h& \Omega_h \end{pmatrix}, 
			&\qquad 
			F_h^{(2)} &= \frac{1}{2n}\begin{pmatrix} \overline{\Omega}_h & \ell_h\overline{\Omega}_h \\ \bar{\ell}_h\overline{\Omega}_h &  \overline{\Omega}_h \end{pmatrix}, \\
			F_h^{(3)} &= \frac{1}{2n}\begin{pmatrix} {\Omega}_h & -\bar{\ell}_h\Omega_h\\ -\ell_h{\Omega}_h & \Omega_h \end{pmatrix}\text{ and}
			&\qquad 
			F_h^{(4)} &= \frac{1}{2n}\begin{pmatrix} \overline{\Omega}_h & -{\ell}_h\overline{\Omega}_h\\ -\bar\ell_h\overline{\Omega}_h & \overline{\Omega}_h \end{pmatrix}.
		\end{aligned}
	\end{equation}
	
	We label the elements of the dihedral group $D_n := \langle a,b \mid a^n = b^2 = 1,\; bab = a^{-1} \rangle$
	as follows. For an integer $u$, if $0 \le u \le n-1$, then the label $u$ denotes the vertex corresponding to $a^u$ in the Cayley graph $\cay(D_n,S)$. Similarly, if $n \le u \le 2n-1$, then $u$ represents the vertex associated with $ba^{u-n}$. Accordingly, we write $G_0 := \{0,\dots,n-1\}$ and	$G_1 := \{n,\dots,2n-1\}.$ 		Recall that $T_n(x)$ is the Chebyshev polynomial of the first kind, and 	$\widetilde{\mu}_i$, $\mu_h$, and $\mu_h^{(i)}$ are eigenvalue of $P$ as defined in 
	\eqref{ev_p_dn1}--\eqref{ev_p_dn3}. For $C\subset\Zl$ and $z\in \Zl$, define $z-C:=\{z-c:c\in C\}$.
	\begin{theorem}\label{pst_dn_1}
		Let $n=2m$ and $D_n$ be the dihedral group. Then the non-normal Cayley graph $\mathrm{Cay}(D_n,S)$ exhibits perfect state transfer
		between the vertices $u$ and $v$  at time $\tau$ if and only if exactly one of the following holds. 
		\begin{description}		
			\item[(A)]
			$(u,v)\in (G_0\times G_0)\cup(G_1\times G_1)$ with $u-v=\pm m$, and the eigenvalues 
			of the discriminant matrix $P$ satisfy
			\begin{enumerate}
				\item[(i)] $T_\tau(\widetilde{\mu}_1)=T_\tau(\widetilde{\mu}_2)=1$,
				$T_\tau(\widetilde{\mu}_3)=T_\tau(\widetilde{\mu}_4)=(-1)^m$,
				\item[(ii)]  $T_\tau(\mu_h)=(-1)^h$ for $1\le h\le m-1$ with  $\eta_h(S_b)=0$, and
				\item[(iii)] $T_\tau(\mu_h^{(1)})=T_\tau(\mu_h^{(2)})=(-1)^h$ for $1\le h\le m-1$ with   $\eta_h(S_b)\neq 0$.
			\end{enumerate}
			\item[(B)]
			$(u,v)\in (G_0\times G_1)\cup(G_1\times G_0)$ with  $\eta_h(S_b)\neq 0$ for each $h$ and $S_b= a^{2|v-u|}S_b^{-1}$,
			and the eigenvalues of the discriminant matrix $P$ satisfy
			\begin{enumerate}
				\item[(i)]  $T_\tau(\widetilde{\mu}_1)=-T_\tau(\widetilde{\mu}_2)=1$,
				$T_\tau(\widetilde{\mu}_3)=-T_\tau(\widetilde{\mu}_4)=(-1)^{u+v}$, and 
				\item[(ii)] $T_\tau(\mu_h^{(1)})=-T_\tau(\mu_h^{(2)})=\ell_h\omega_n^{|v-u|h}$ for $1\le h\le m-1$.
			\end{enumerate}
		\end{description}		
	\end{theorem}
	\begin{proof}
		\noindent\textbf{(A)} Here $(u,v)\in (G_0\times G_0)\cup(G_1\times G_1)$. Using \eqref{esp_1}--\eqref{esp_3} in \eqref{sd_tp_dn}, the $(u,v)$-entry of $T_\tau(P)$ is given by
		\begin{equation}\label{pst_dn_eq}
			\begin{aligned}
				T_\tau(P)_{u,v}&=\frac{1}{2n}\left(T_\tau(\widetilde{\mu}_1)+T_\tau(\widetilde{\mu}_2)+(-1)^{u+v}\left(T_\tau(\widetilde{\mu}_3)+T_\tau(\widetilde{\mu}_4)\right)\right)\\
				&\quad +\frac{1}{n}\sum_{\substack{h=1\\ \eta_h(S_b)=0}} ^{m-1}\left(\omega_n^{(v-u)h}+\omega_n^{(u-v)h}\right)T_\tau(\mu_h)\\
				&\quad +\frac{1}{2n}\sum_{\substack{h=1\\ \eta_h(S_b)\neq 0}} ^{m-1}\left(\omega_n^{(u-v)h}+\omega_n^{(v-u)h}\right)T_\tau(\mu_h^{(1)})\\
				&\quad +\frac{1}{2n}\sum_{\substack{h=1\\ \eta_h(S_b)\neq 0}} ^{m-1}\left(\omega_n^{(u-v)h}+\omega_n^{(v-u)h}\right)T_\tau(\mu_h^{(2)}).
			\end{aligned}
		\end{equation}
		Applying the triangle inequality to \eqref{pst_dn_eq} and using Corollary~\ref{minimum}, $|T_\tau(P)_{u,v}|\leq 1$, and equality occurs only when all summands have the same argument. Thus $|T_\tau(P)_{u,v}|=1$ if and only if
		\begin{equation}\label{pst_dn_eq3}
			\begin{aligned}	
				T_\tau(\widetilde{\mu}_1)&=T_\tau(\widetilde{\mu}_2)\\
				&=(-1)^{u+v}T_\tau(\widetilde{\mu}_3)=(-1)^{u+v}T_\tau(\widetilde{\mu}_4)\\
				&=\omega_n^{(u-v)h}T_\tau(\mu_h)=\omega_n^{(v-u)h}T_\tau(\mu_h)\quad\text{for $1\leq h\leq m-1$ with $\eta_h(S_b)=0$}\\
				&=\omega_n^{(u-v)h}T_\tau(\mu_h^{(1)})=\omega_n^{(v-u)h}T_\tau(\mu_h^{(1)})\quad\text{for $1\leq h\leq m-1$ with $\eta_h(S_b)\neq0$}\\
				&=\omega_n^{(u-v)h}T_\tau(\mu_h^{(2)})=\omega_n^{(v-u)h}T_\tau(\mu_h^{(2)})\quad\text{for $1\leq h\leq m-1$ with $\eta_h(S_b)\neq0$}.
			\end{aligned}
		\end{equation}
		
		Suppose $\mathrm{Cay}(D_n,S)$ exhibits perfect state transfer between the vertices $u$ and $v$  at time $\tau$, that is $T_\tau(P)_{u,v}=1$. Hence \eqref{pst_dn_eq3} holds. Therefore $\omega_n^{(u-v)h}=\omega_n^{(v-u)h}$ for $1\leq h\leq m-1$. Since either $0 \le u,v \le n-1$ or $n \le u,v \le 2n-1$, it follows that $u-v=\pm m $. Note that $\widetilde{\mu}_1=1$. Thus $T_\tau(\widetilde{\mu}_1)=1$. Therefore by \eqref{pst_dn_eq3},  Conditions (i)--(iii) of (A) follow.

		Conversely, if (A) holds, then \eqref{pst_dn_eq3} is satisfied. Hence $|T_\tau(P)_{u,v}|=1$, and therefore perfect state transfer occurs between $u$ and $v$ at time~$\tau$.

		\noindent\textbf{(B)} We consider two cases according as $(u,v)\in G_0\times G_1$ or $(u,v)\in G_1\times G_0$.
		
		\noindent\textbf{Case 1.}  First assume that $(u,v)\in G_0\times G_1$. Using \eqref{esp_1}--\eqref{esp_3} in \eqref{sd_tp_dn}, the $(u,v)$-entry of $T_\tau(P)$ is given by
		\begin{equation}\label{pst_dn_eq4}
			\begin{aligned}
				T_\tau(P)_{u,v}&=\frac{1}{2n}\left(T_\tau(\widetilde{\mu}_1)-T_\tau(\widetilde{\mu}_2)+(-1)^{u+v} T_\tau(\widetilde{\mu}_3)+(-1)^{u+v+1}T_\tau(\widetilde{\mu}_4)\right)\\
				&\quad +\frac{1}{2n}\sum_{\substack{h=1\\ \eta_h(S_b)\neq 0}} ^{m-1}\left(\bar\ell_h\omega_n^{(u-v)h}+\ell_h\omega_n^{(v-u)h}\right)T_\tau(\mu_h^{(1)})\\
				&\quad -\frac{1}{2n}\sum_{\substack{h=1\\ \eta_h(S_b)\neq 0}} ^{m-1}\left(\bar\ell_h\omega_n^{(u-v)h}+\ell_h\omega_n^{(v-u)h}\right)T_\tau(\mu_h^{(2)}).
			\end{aligned}
		\end{equation}
		If $\eta_h(S_b)=0$ for some $h$, then the triangle inequality implies that $|T_\tau(P)_{u,v}|<1$. Therefore $|T_\tau(P)_{u,v}|=1$ if and only if $\eta_h(S_b)\neq 0$ for $1\leq h\leq m-1$ and 
		\begin{equation}\label{pst_dn_eq5}
			\begin{aligned}	
				T_\tau(\widetilde{\mu}_1)&=-T_\tau(\widetilde{\mu}_2)\\
				&=(-1)^{u+v}T_\tau(\widetilde{\mu}_3)=(-1)^{u+v+1}T_\tau(\widetilde{\mu}_4)\\
				&=\bar\ell_h\omega_n^{(u-v)h}T_\tau(\mu_h^{(1)})=\ell_h\omega_n^{(v-u)h}T_\tau(\mu_h^{(1)})\\
				&=-\bar\ell_h\omega_n^{(u-v)h}T_\tau(\mu_h^{(2)})=-\ell_h\omega_n^{(v-u)h}T_\tau(\mu_h^{(2)}).
			\end{aligned}
		\end{equation}
		
		Suppose $\mathrm{Cay}(D_n,S)$ exhibits perfect state transfer between the vertices $u$ and $v$  at time $\tau$. By \eqref{pst_dn_eq5}, we find $\bar\ell_h\omega_n^{(u-v)h}={\ell_h}\omega_n^{(v-u)h}$, and therefore $\ell_h^2=\omega_n^{2(u-v)h}$  for $1\leq h\leq m-1$. Hence
		we find
		\begin{equation*}
			{\eta_h(S_b^{-1})}=\omega_n^{2(u-v)h}\eta_h(S_b).
		\end{equation*}
		This implies that
		\[
		\sum_{s\in S_b^*}\omega_n^{-hs}
		=
		\omega_n^{2(u-v)h}\sum_{s\in S_b^*}\omega_n^{hs}.
		\]
		Multiplying both sides by $\omega_n^{(v-u)h}$ gives
		\[
		\sum_{s\in S_b^*}\omega_n^{h(v-u-s)}
		=
		\sum_{s\in S_b^*}\omega_n^{h(s+u-v)}.
		\]
		Define $E:=\{v-u-s : s\in S_b^*\}$ and $F:=\{s+u-v : s\in S_b^*\}$.
		Then
		\begin{equation}\label{ranjan}
			\sum_{a\in E}\omega_n^{ha}=\sum_{b\in F}\omega_n^{hb}
			\quad \text{for }1\leq h\leq m-1.
		\end{equation}
		Note that
		\[
		\frac{1}{n}\sum_{h=0}^{n-1} \omega_n^{h(x-y)} =
		\begin{cases}
			1 & \text{if } x=y,\\
			0 & \text{if } x\neq y.
		\end{cases}
		\]
		For any $x \in\Zl_n$, we compute
		\[
		\frac{1}{n}\sum_{h=0}^{n-1}
		\left(\sum_{a\in E}\omega_n^{ha}\right)\omega_n^{-hx}
		=
		\frac{1}{n}\sum_{a\in E}\sum_{h=0}^{n-1} \omega_n^{h(a-x)}
		=
		\sum_{a\in E}
		\left(
		\frac{1}{n}\sum_{h=0}^{n-1} \omega_n^{h(a-x)}
		\right)
		= \mathbf{1}_E(x),
		\]
		where $\mathbf{1}_E$ denotes the indicator function of $E$. Using \eqref{ranjan}, we have
		\[
		\mathbf{1}_E(x)
		=
		\frac{1}{n}\sum_{h=0}^{n-1}
		\left(\sum_{b\in F}\omega_n^{hb}\right)\omega_n^{-hx}
		= \mathbf{1}_F(x).
		\]
		Hence $\mathbf{1}_E(x)=\mathbf{1}_F(x)$ for all $x \in \Zl_n$, and therefore $E=F$.  Hence for every $s\in S_b^*$ there exists $t\in S_b^*$ such that  $t= 2(v-u)-s $ in $\Zl_n$. Hence $S_b^* = 2(v-u) - S_b^*$ in $\Zl_n$, that is, $S_b=a^{2(v-u)}S_b^{-1}$. Now, by \eqref{pst_dn_eq5},  Conditions (i) and (ii) of (B) follow.

		Conversely, if (B) holds, then the condition $S_b=a^{2(v-u)}S_b^{-1}$, that is,  $S_b^* = 2(v-u)- S_b^*$ gives that 
		\[
		\eta_h(S_b)
		=\sum_{s\in S_b^*}\omega_n^{hs}
		=\sum_{s\in S_b^*}\omega_n^{h(2(v-u)-s)}.
		\]
		Therefore 
		\[
		\eta_h(S_b)
		=\frac{1}{2}\sum_{s\in S_b^*}\Bigl(\omega_n^{hs}+\omega_n^{h(2(v-u)-s)}\Bigr)=\omega_n^{h(v-u)}\kappa_h,
		\]
		where $$\kappa_h=\frac{1}{2}\sum_{s\in S_b^*}
		\left(\omega_n^{h(u-v+s)}+\omega_n^{h(v-u-s)}\right)\in \Rl.$$
		Thus $\ell_h^2
		={\overline{\eta_h(S_b)}}/{\eta_h(S_b)}
		=\omega_n^{2h(u-v)},$
		giving $\ell_h=\pm\omega_n^{h(u-v)}$. Consequently, $\ell_h\omega_n^{(v-u)h}=\pm1
		=\bar{\ell_h}\omega_n^{(u-v)h}$. Using Conditions (i) and (ii) of (B) together with \eqref{pst_dn_eq5}, we conclude that $T_\tau(P)_{u,v}=1$, that is,  perfect state transfer occurs between $u$ and $v$ at time $\tau$.
		
		\noindent\textbf{Case 2.} Now assume that $(u,v)\in G_1\times G_0$. In this case, we interchange the roles of $u$ and $v$, and proceed as in Case 1.

		Note that Condition (i) of (A) and Condition (i) of (B) cannot hold altogether at the same time $\tau$. Therefore, if perfect state transfer occurs on $\cay(D_n,S)$ at time $\tau$, then exactly one of (A) and (B) holds.
	\end{proof}
	
	\begin{theorem}\label{pst_dn_2}
		Let $n=2m$ with $n\geq 4$ and $D_n$ be the dihedral group. Then the normal Cayley graph $\mathrm{Cay}(D_n,S)$ exhibits perfect state transfer
		between the vertices $u$ and $v$ at time $\tau$ if and only if 
		$(u,v)\in (G_0\times G_0)\cup(G_1\times G_1)$ with $u-v=\pm m$, and the eigenvalues 
		of the discriminant matrix $P$ satisfy
		\begin{enumerate}
			\item[(i)] $T_\tau(\widetilde{\mu}_1)=T_\tau(\widetilde{\mu}_2)=1$, 
			$T_\tau(\widetilde{\mu}_3)=T_\tau(\widetilde{\mu}_4)=(-1)^m$, and
			\item[(ii)]  $T_\tau(\mu_h)=(-1)^h$ for $1\le h\le m-1$.
		\end{enumerate}
	\end{theorem}
	\begin{proof}
		Suppose that $\cay(D_n,S)$ exhibits perfect state transfer between the vertices $u$ and $v$ at time $\tau$.  By Lemma~\ref{pst_cayley} and the fact that $Z(D_n)=\{1,a^m\}$, we find $T_\tau(P)=\mathcal{R}_{a^m}$.  Note that if $(u,v)\in G_0\times G_1$, then $u$ corresponds to $a^r$ and $v$ corresponds to $ba^s$ for some $r$ and $s$. Therefore the $uv$-th entry of $R_{a^m}$ is $0$, that is, $\evt T_\tau(P)\eu=0$. Similarly, if $(u,v)\in G_1\times G_0$, then also $\evt T_\tau(P)\eu=0$. Thus perfect state transfer does not occur if $(u,v)\in (G_0\times G_1)\cup(G_1\times G_0)$. Hence $(u,v)\in (G_0\times G_0)\cup(G_1\times G_1)$. Since the graph is normal, we have $\eta_h(S_b)=0$ for each $h$. The result now follows by arguments analogous to those in the proof of Case~1 in Theorem~\ref{pst_dn_1}.
	\end{proof}
	We now assume that $n$ is odd.  Then $D_n$ has only two one-dimensional representations. 
	Accordingly, the eigenvalues of the discriminant $P$ of $\cay(D_n,S)$ corresponding to these representations are
	\[
	\widetilde{\mu}_i=\frac{\widetilde{\lambda}_i}{d}, \quad i\in\{1,2\},
	\]
	with corresponding eigenvectors as in Theorem~\ref{ev1_dn}. For the two-dimensional representations, the eigenvalues and eigenvectors of $P$ are the same as in the even case described in Theorem~\ref{ev2_dn}. 
	\begin{theorem}\label{pst_dn_3}
		Let $n=2m+1$ and $D_n$ be the dihedral group. Then the non-normal Cayley graph $\mathrm{Cay}(D_n,S)$ exhibits perfect state transfer
		between the vertices $u$ and $v$ at time $\tau$ if and only if  $(u,v)\in (G_0\times G_1)\cup(G_1\times G_0)$ with  $\eta_h(S_b)\neq 0 $ for each $h$ and  $S_b = a^{2|v-u|} S_b^{-1}$,
		and the eigenvalues of the discriminant matrix $P$ satisfy
		\begin{enumerate}
			\item[(i)]  $T_\tau(\widetilde{\mu}_1)=-T_\tau(\widetilde{\mu}_2)=1$, and
			\item[(ii)] $T_\tau(\mu_h^{(1)})=-T_\tau(\mu_h^{(2)})=\ell_h\omega_n^{|v-u|h}$ for $1\le h\le m$.
		\end{enumerate}	 		
	\end{theorem}
	\begin{proof}
		We first show that $(u,v)\notin (G_0\times G_0)\cup(G_1\times G_1)$. Suppose, on the contrary, that $(u,v)\in (G_0\times G_0)\cup(G_1\times G_1)$. Then applying the same argument as in the proof of Theorem~\ref{pst_dn_1}, we obtain $\omega_n^{(u-v)h}=\omega_n^{(v-u)h}$ for $1\leq h\leq m$. Taking $h=1$, it follows that $2(u-v)\equiv 0\pmod n$. Since $n$ is odd, this implies that $u=v$, a contradiction. Therefore $(u,v)\in (G_0\times G_1)\cup(G_1\times G_0)$. The remainder of the proof now follows by the same argument as in the proof of Theorem~\ref{pst_dn_1}.
	\end{proof}
	\begin{theorem}\label{pst_dn_4}
		Let $n$ be odd, and $D_n$ be the dihedral group. Then the normal Cayley graph $\mathrm{Cay}(D_n,S)$ does not exhibit perfect state transfer.
	\end{theorem}
	\begin{proof}
		Suppose, if possible, the normal Cayley graph $\cay(D_n,S)$ exhibits perfect state transfer at time $\tau$. Then by Lemma~\ref{pst_cayley}, we  have $T_\tau(P)=\mathcal{R}_z$ for some $z\in Z(D_n)$ of order $2$. However, $Z(D_n)=\{1\}$ whenever $n$ is odd. Thus the result follows.
	\end{proof}

	\section{Complete Classification of perfect state transfer} 
	We recall a few basic notions from algebraic number theory that will be used
	throughout this section. A complex number $\alpha$ is an \emph{algebraic number} if it is a 	root of a nonzero polynomial in $\mathbb{Q}[x]$. The unique monic irreducible polynomial $m_\alpha(x)\in\Ql[x]$ having $\alpha$ as a root is called the \emph{minimal polynomial} of $\alpha$ over $\mathbb{Q}$.  The roots of the minimal polynomial $m_\alpha(x)$ are called the   	\emph{algebraic conjugates} of $\alpha$ over $\mathbb{Q}$.  
	
	An algebraic number $\alpha$ is called an \emph{algebraic integer} if it is a root of a monic polynomial with coefficients in $\mathbb{Z}$. We recall that every rational algebraic integer is an integer. A \emph{number field} is a finite field extension $K$ of $\mathbb{Q}$. The set of all algebraic integers contained in a number field $K$ is denoted by $\mathcal{O}_K$ and is called the \emph{ring of integers} of $K$. 
	
	Let $[K:\Ql]=m$. An \emph{integral basis} of $K$ is 	a $\mathbb{Z}$-basis $\beta_1,\ldots,\beta_m$ 	of its ring of integers $\mathcal{O}_K$. Thus $\mathcal{O}_K=\mathbb{Z}\beta_1\oplus\cdots\oplus\mathbb{Z}\beta_m$.
	Equivalently, every $\alpha\in\mathcal{O}_K$ can be written uniquely as $\alpha=a_1\beta_1+\cdots+a_m\beta_m$ for some
	$\{a_1,\ldots,a_m\}\subset\mathbb{Z}$.  The field $\mathbb{Q}(\omega_n)$ is called the \emph{$n$th cyclotomic field}. 
	
	A \emph{splitting field} of a nonconstant polynomial $f(x)$ over $\Ql$ is the smallest field containing $\Ql$ and all the roots of $f$. An \emph{automorphism of $K$ over $\Ql$} is a field isomorphism $\sigma:K\to K$ fixing every element of $\Ql$. If $K$ is a splitting field over $\Ql$, then every algebraic conjugate of $\alpha\in K$ is of the form $\sigma(\alpha)$ for some automorphism $\sigma$ of $K$ over $\Ql$. Moreover, an element of $K$ belongs to $\Ql$ if and only if it is fixed by every such automorphism.

	\begin{lemma}[{\cite[Proposition~10.2]{ant_neukirch}}]\label{integral_basis}
		Let $K=\mathbb{Q}(\omega_n)$. Then ${\mathcal{O}_K=\mathbb{Z}[\omega_n]}$ and $1,\omega_n,\ldots,\omega_n^{\varphi(n)-1}$ is an integral basis of $K$.
	\end{lemma}

	\begin{corollary}\label{integral_integer}
		Let $\alpha\in\Ql(\omega_n)$ be an algebraic integer. If  $\alpha=q_0+q_1\omega_n+\cdots+q_{\varphi(n)-1}\omega_n^{\varphi(n)-1}$,  where  $q_i\in\Ql$ for $0\leq i\leq \varphi(n)-1$, then ${q_i\in\Zl}$ for $0\leq i\leq \varphi(n)-1$.
	\end{corollary}
	\begin{proof}
		Since $\alpha$ is an algebraic integer, Lemma~\ref{integral_basis} gives $a_i\in\Zl$ for $0\leq i\leq \varphi(n)-1$ such that $\alpha=a_0+a_1\omega_n+\cdots+a_{\varphi(n)-1}\omega_n^{\varphi(n)-1}$. Therefore $\sum_{j=0}^{\varphi(n)-1}(q_j-a_j)\omega_n^j=0$, and since $1,\omega_n,\ldots,\omega_n^{\varphi(n)-1}$ are linearly independent	over $\mathbb{Q}$, we get $q_j=a_j$ for $0\leq j\leq\varphi(n)-1$. Thus each $q_j$ is an integer.
	\end{proof}
	
	\begin{lemma}\label{algebraic_pst}
		Let $\mu$ be an eigenvalue of the discriminant of a Cayley graph $G$. If $G$ exhibits perfect state transfer, then $2\mu$ is an algebraic integer. Further, if $\mu\in\Ql$, then $\mu\in\{\pm1,\ \pm1/2,\ 0\}$.
	\end{lemma}
	\begin{proof}
		Suppose $G$ exhibits perfect state transfer at time $\tau$. Then by Corollary~\ref{minimum}, $T_\tau(\mu)=\pm1$. Hence $\mu=\cos\frac{j\pi}{\tau}$ for some integer $j$. This implies that $2\mu=\omega_{2\tau}^j+\omega_{2\tau}^{-j}$. Hence $2\mu$ is an algebraic integer.
		
		Suppose $\mu\in\Ql$, and hence $2\mu\in\Ql$. Since $2\mu$ is a rational algebraic integer, $2\mu\in\Zl$. By Lemma~\ref{reg}, $|\mu|\leq1$. Hence $\mu\in\{\pm1,\ \pm1/2,\ 0\}$.
	\end{proof}
	
	For a function $f:\mathbb{Z}_n\to\mathbb{Z}$, define 
	\begin{equation}\label{crb2212}
		\widehat f(h)=\sum_{k=0}^{n-1}f(k)\omega_n^{hk}\quad\text{for}\quad h\in\mathbb{Z}_n.
	\end{equation}
	
	\begin{lemma}\label{divisibility}
		Let $q=p^a$, where $p$ is a prime and $a\geq1$. Let $f:\Zl_n\to\Zl$ be a function and $\widehat f(h)/q$ be an algebraic integer for each $h\in\mathbb{Z}_n$. Then the following holds.
		\begin{enumerate}
			\item[(i)] If $p\nmid n$, then $q\mid f(k)$ for each $k$.
			\item[(ii)] If $p\mid n$ and $\max f-\min f<q$, then $\widehat f(h)=0$ for each $h\in\Zl_n$ with $p\nmid h$.
		\end{enumerate}
	\end{lemma}	 
	\begin{proof} 
		\begin{enumerate}
			\item[(i)] By discrete Fourier inversion,
			\[
			\frac{nf(k)}{q}
			=\sum_{h=0}^{n-1}\frac{\widehat f(h)}{q}\omega_n^{-hk}.
			\]
			Therefore $nf(k)/q$ is a rational algebraic integer, and hence an integer. Thus $q\mid nf(k)$. Since $\gcd(n,q)=1$, we obtain $q\mid f(k)$ for each $k$.
			
			\item[(ii)] Write $n=p^sm$, where $p\nmid m$ and
			$s\geq1$. Let $\zeta=\omega_n^m$. For $0\leq r\leq m-1$, define $\beta_r$ by
			\[
			\beta_r=\sum_{j=0}^{m-1}\frac{\widehat f(1+jp^s)}{q}\omega_n^{-r(1+jp^s)}.
			\]
			Note that $\beta_r$ is an algebraic integer. Using \eqref{crb2212}, we obtain
			\[
			\begin{aligned}
				\beta_r
				&=\frac1q\sum_{k=0}^{n-1}f(k)\omega_n^{k-r}
				\sum_{j=0}^{m-1}\omega_n^{jp^s(k-r)}\\
				&=\frac{m}{q}
				\sum_{\substack{k=0\\ m\divides (k-r)}}^{n-1} f(k)\omega_n^{k-r}\\
				&=\frac{m}{q}
				\sum_{j=0}^{p^s-1}f(r+mj)\zeta^j.
			\end{aligned}
			\]
			Let $\ell=p^{s-1}$ and $a_j=f(r+mj)$. Since $p^s=p\ell$, each index $k$ satisfying  $0\leq k\leq p^s-1$ can be written uniquely as $k=i+j\ell$, where $0\leq i\leq \ell-1$ and $0\leq j\leq p-1$. Using $\zeta^{(p-1)\ell}=-(1+\zeta^\ell+\cdots+\zeta^{(p-2)\ell})$, we obtain
			\[
			\beta_r=\frac mq\sum_{k=0}^{p^s-1}a_k\zeta^k=\frac mq\sum_{i=0}^{\ell-1}\sum_{j=0}^{p-1}a_{i+j\ell}\zeta^{i+j\ell}=\sum_{i=0}^{\ell-1}\sum_{j=0}^{p-2}\frac{m\bigl(a_{i+j\ell}-a_{i+(p-1)\ell}\bigr)}q\zeta^{i+j\ell}.
			\]
			The powers appearing in the preceding sum are precisely $1,\zeta,\ldots,\zeta^{\varphi(p^s)-1}$, which form an integral basis of $\mathbb{Q}(\zeta$). Thus by Corollary~\ref{integral_integer},
			\[
			\frac{m(a_{i+j\ell}-a_{i+(p-1)\ell})}{q}
			\in\mathbb{Z}.
			\]
			Therefore $q$ divides $a_{i+j\ell}-a_{i+(p-1)\ell}$.	Since $\left|a_{i+j\ell}-a_{i+(p-1)\ell}\right|\leq \max f-\min f<q$, it follows that $a_i=a_{i+\ell}=\cdots=a_{i+(p-1)\ell}$ for $0\leq i\leq\ell-1$. Writing $t=i+j\ell$, with $0\leq i\leq \ell-1$ and $0\leq j\leq p-1$, gives $a_{t+\ell}=a_t$, with subscripts taken modulo $p^s$. Since $m\ell=n/p$, writing $k=r+mt$, where $0\leq r\leq m-1$ and $0\leq t\leq p^s-1$, yields $f(k+n/p)=a_{t+\ell}=a_t=f(k)$ for every $k\in\Zl_n$. For $h\in\Zl_n$ with $p\nmid h$, we obtain
			\[
			\widehat f(h)=\sum_{k=0}^{\frac np-1}\sum_{j=0}^{p-1}f(k+jn/p)\omega_n^{h(k+jn/p)}=\sum_{k=0}^{\frac np-1}f(k)\omega_n^{hk}\sum_{j=0}^{p-1}\omega_n^{hjn/p}=0.\qedhere
			\]
		\end{enumerate}
	\end{proof}

	\subsection{Reduction to valencies two and four}
	In this section, we show that the Cayley graph $\cay(D_n,S)$ that exhibits perfect state transfer has valency $2$ or $4$.  By vertex-transitivity, it suffices to consider perfect state 	transfer from the vertex $1$ to another vertex $v$ on $\cay(D_n,S)$. Recall the notation in \eqref{crb2216} and $\eta_h(C)$.
	\begin{lemma}\label{valency_reflection}
		Let $n\geq3$, and let $S\subseteq D_n$ satisfying $S=S^{-1}$ and $\langle S\rangle=D_n$. If $\cay(D_n,S)$ exhibits perfect state transfer from $1$ to $ba^j$ for some $j$, then $d\in\{2,4\}$.
	\end{lemma}
	\begin{proof}
		Suppose $\cay(D_n,S)$ exhibits perfect state transfer from $1$ to
		$ba^j$ at time $\tau$. Then by Theorems~\ref{pst_dn_1}--\ref{pst_dn_4}, $S$ is non-normal.
		By Theorem~\ref{pst_dn_1}(B) or Theorem~\ref{pst_dn_3}, we have
		$S_b^*=2j-S_b^*$.
		Thus $C:=S_b^*-j$ satisfies $C=-C$. Define
		\[
		C_h=\sum_{x\in C}\omega_n^{hx}
		\quad\text{and}\quad
		f(k)=\mathbf1_{S_a^*}(k)+\mathbf1_C(k).
		\]
		Then $C_h$ is real, $f(k)\in\{0,1,2\}$,
		$\eta_h(S_b)=\omega_n^{jh}C_h$, and
		$\widehat f(h)=\eta_h(S_a)+C_h$.
		By Theorem~\ref{pst_dn_1}(B) or Theorem~\ref{pst_dn_3}, we have
		\[
		T_\tau(\mu_h^{(1)})
		=-T_\tau(\mu_h^{(2)})
		=\ell_h\omega_n^{jh}
		=\frac{\overline{\eta_h(S_b)}}{|\eta_h(S_b)|}\omega_n^{jh}
		=\frac{C_h}{|C_h|}
		\quad\text{for }1\leq h\leq
		\left\lfloor\frac{n-1}{2}\right\rfloor.
		\]
		If $C_h>0$, then $\widehat f(h)/d=\mu_h^{(1)}$, and if $C_h<0$,
		then $\widehat f(h)/d=\mu_h^{(2)}$.
		Consequently, in both cases,
		\begin{equation}\label{crb1047}
			T_\tau\left(\frac{\widehat f(h)}d\right)=1
			\quad\text{for }1\leq h\leq
			\left\lfloor\frac{n-1}{2}\right\rfloor.
		\end{equation}
		By Lemma~\ref{algebraic_pst}, $\widetilde\mu_2\in\{0,\pm\frac12,\pm1\}$. This implies that $4\nmid\tau$. Indeed, if $4\mid\tau$, then $T_\tau(\widetilde\mu_2)\neq-1$, contradicting Theorem~\ref{pst_dn_1}(B) or Theorem~\ref{pst_dn_3}. Hence $T_\tau(0)\neq1$. It follows from \eqref{crb1047} that
		\begin{equation}\label{crb1107}
			\widehat f(h)\neq0
			\quad\text{for }1\leq h\leq
			\left\lfloor\frac{n-1}{2}\right\rfloor.
		\end{equation}

		By Lemma~\ref{algebraic_pst}, $2\widehat f(h)/d$ is an algebraic integer for $1\leq h\leq\lfloor(n-1)/2\rfloor$.
		At $h=0$, we have ${\widehat f(0)}/d ={(d_a+d_b)}/{d}=1=\widetilde\mu_1$. If $n$ is even, then $C=S_b^*-j$ gives
		\[
		\frac{\widehat f(n/2)}d
		=\frac{\eta_{n/2}(S_a)+(-1)^j\eta_{n/2}(S_b)}d
		=
		\begin{cases}
			\widetilde\mu_3&\text{if $j$ is even},\\
			\widetilde\mu_4&\text{if $j$ is odd}.
		\end{cases}
		\]
		Finally, since $\widehat f(n-h)=\widehat f(h)$, we have $2\widehat f(h)/d$ is an algebraic integer for each $h\in\mathbb Z_n$.
		
		Put $d_0=d/\gcd(d,2)$. Since $\widehat f(h)$ and $2\widehat f(h)/d$ are algebraic integers, so is $\widehat f(h)/d_0$. Indeed, this is immediate for even $d$, and	for odd $d$, this follows from
		\[
		\frac{\widehat f(h)}{d_0}=\frac{\widehat f(h)}{d}
		=\frac{d+1}{2}\frac{2\widehat f(h)}d-\widehat f(h).
		\]
		Suppose that $d_0>2$. choose a prime-power divisor $q=p^a>2$ of $d_0$. Then $\widehat f(h)/q=(d_0/q)(\widehat f(h)/d_0)$ is an algebraic integer for each $h\in\Zl_n$. If $p\nmid n$, Lemma~\ref{divisibility}(i) gives $q\mid f(k)$ for every $k$. Since $q>2$ and $f(k)\in\{0,1,2\}$, this forces $f\equiv0$, a contradiction. If $p\mid n$, then $\max f-\min f\leq2<q$, and hence Lemma~\ref{divisibility}(ii) gives $\widehat f(1)=0$, contrary to \eqref{crb1107}. Hence $d_0\leq2$, which gives $d\in\{1,2,4\}$. Since the graph is connected, $d\in\{2,4\}$.
	\end{proof}

	For $X\subseteq\Zl_n$ and $0\leq k\leq n-1$, define $\Delta_X(k)=\left|\{(x,y)\in X\times X:x-y\equiv k\pmod n\}\right|$. Note that $\Delta_X(k)=|(X+k)\cap X|$. Thus $\Delta_X(k)\leq|X|$, $\Delta_X(0)=|X|$, and $\Delta_X(k)=|X|$ if and only if $X+k=X$. For $n=2m$ and  $\alpha\in\{0,1\}$, we obtain
	\begin{align}
		\sum_{\substack{0\leq h<n\\h\equiv\alpha\pmod2}}
		\left|\sum_{x\in X}\omega_n^{hx}\right|^2
		&=\sum_{\substack{0\leq h<n\\h\equiv\alpha\pmod2}}\sum_{k=0}^{n-1}\Delta_X(k)\omega_n^{hk}\nonumber\\
		&=\sum_{k=0}^{n-1}\Delta_X(k)\omega_n^{\alpha k}
		\sum_{r=0}^{m-1}\omega_n^{2rk}\nonumber\\
		&=m\bigl(|X|+(-1)^\alpha \Delta_X(m)\bigr). \label{crb1316}
	\end{align}

	We first establish an inequality that will be used in the subsequent results. Suppose that $n=2m$ and perfect state transfer occurs from $1$ to $a^m$ at time $\tau$. Let $c=\Delta_{S_a^*}(m)+\Delta_{S_b^*}(m)$. Then $0\leq c\leq d$. For $\alpha\in\{0,1\}$, let $\mathcal E_\alpha$ be the multiset of eigenvalues $\mu$ of $P$ satisfying $T_\tau(\mu)=(-1)^\alpha$, counted with their multiplicities. By Theorem~\ref{pst_dn_1}(A) or Theorem~\ref{pst_dn_2}, $\{\widetilde{\mu}_1,\widetilde{\mu}_2\}\subseteq\mathcal E_0$ and $\{\widetilde{\mu}_3,\widetilde{\mu}_4\}\subseteq\mathcal E_\alpha$ if $m\equiv\alpha\pmod2$. Moreover, for $1\leq h\leq m-1$ with $h\equiv\alpha\pmod 2$,
	\[
	\begin{cases}
		\mu_h\in\mathcal E_\alpha
		& \text{if } \eta_h(S_b)=0,\\
		\{\mu_h^{(1)},\mu_h^{(2)}\}\subseteq\mathcal E_\alpha
		& \text{if } \eta_h(S_b)\neq 0.
	\end{cases}
	\]
	Counting eigenvalues with their multiplicities, we have $|\mathcal E_\alpha|=2m$.

	For $1\leq h\leq m-1$, the eigenvalue formulas give
	\[
	\frac4{d^2}
	\bigl(|\eta_h(S_a)|^2+|\eta_h(S_b)|^2\bigr)
	=
	\begin{cases}
		2\bigl((\mu_h^{(1)})^2+(\mu_h^{(2)})^2\bigr)
		&\text{if }\eta_h(S_b)\neq0,\\
		4\mu_h^2
		&\text{if }\eta_h(S_b)=0.
	\end{cases}
	\]
	For the one-dimensional eigenvalues, we have
	\[
	\begin{aligned}
		\widetilde\mu_1^2+\widetilde\mu_2^2
		&=\frac2{d^2}
		\bigl(|\eta_0(S_a)|^2+|\eta_0(S_b)|^2\bigr),\\
		\widetilde\mu_3^2+\widetilde\mu_4^2
		&=\frac2{d^2}
		\bigl(|\eta_m(S_a)|^2+|\eta_m(S_b)|^2\bigr).
	\end{aligned}
	\]
	Since $|\eta_{n-h}(X)|=|\eta_h(X)|$ for $X\in\{S_a,S_b\}$ and $n-h\equiv h\pmod2$, we pair these indices. Including the one-dimensional eigenvalues and applying \eqref{crb1316}, we obtain
	\[
	\sum_{\mu\in\mathcal E_\alpha}\mu^2=\frac2{d^2}
	\sum_{\substack{0\leq h<n\\h\equiv\alpha\pmod2}}
	\bigl(|\eta_h(S_a)|^2+|\eta_h(S_b)|^2\bigr)=\frac{2m\bigl(d+(-1)^\alpha c\bigr)}{d^2}.
	\]
	Suppose that each eigenvalue of $\mathcal E_\alpha$ is nonzero, and set
	\[
	\nu_\alpha=\prod_{\mu\in\mathcal E_\alpha}2\mu.
	\]
	By Lemma~\ref{algebraic_pst}, $ \nu_\alpha$ is a nonzero algebraic integer. Let $K$ be a splitting field of the characteristic
	polynomial of $P$, and let $\sigma$ be an automorphism of $K$ over $\Ql$. Since $P$ has rational entries, $\sigma$ permutes its eigenvalues with their multiplicities. Moreover, $T_\tau(x)\in\mathbb Z[x]$, and hence for $\mu\in\mathcal E_\alpha$, $T_\tau(\sigma(\mu))=\sigma(T_\tau(\mu))=(-1)^\alpha$. Thus $\sigma$ permutes $\mathcal E_\alpha$ and hence $\sigma(\nu_\alpha)=\nu_\alpha$. Therefore $ \nu_\alpha\in\Ql$, and so $ \nu_\alpha\in\mathbb Z\setminus\{0\}$. Hence $| \nu_\alpha|\geq1$.

	Applying the arithmetic--geometric mean inequality, we obtain
	\begin{equation}\label{equality_1334}
		\begin{aligned}
			1\leq | \nu_\alpha|^{1/m}=\left(\prod_{\mu\in\mathcal E_\alpha}(2\mu)^2\right)^{\frac{1}{2m}}\leq\frac1{2m}\sum_{\mu\in\mathcal E_\alpha}(2\mu)^2
			=\frac{4\bigl(d+(-1)^\alpha c\bigr)}{d^2}.
		\end{aligned}
	\end{equation}
	If equality occurs throughout in \eqref{equality_1334}, then $(2\mu)^2=1$, and hence $\mu=\pm\frac12$ for each  $\mu\in\mathcal E_\alpha$. Note that \eqref{equality_1334} holds under the assumption that each eigenvalue of $\mathcal E_\alpha$ is nonzero.

	\begin{lemma}\label{central_time}
		Let $n=2m\geq4$, and let $S\subseteq D_n\setminus\{1\}$ satisfying $S=S^{-1}$, $\langle S\rangle= D_n$, and $d\geq4$. If $\cay(D_n,S)$ exhibits perfect state transfer from $1$ to $a^m$ at the minimum time $\tau$,	then $\tau\equiv2\pmod4$.
	\end{lemma}
	\begin{proof}
		By Corollary~\ref{minimum}, $\cay(D_n,S)$ is $2\tau$-periodic. Since $d\geq4$, Lemma~\ref{kubota_regular} implies that $\tau$ is even. If $4\mid\tau$, then $T_\tau(0)=1$. Since $T_\tau(\mu)=-1$ for $\mu\in\mathcal E_1$, it follows that $\mu\neq0$ for each $\mu\in\mathcal E_1$. Therefore by \eqref{equality_1334}, we have
		\[
		1\leq | \nu_1|^{1/m}
		\leq\frac{4(d-c)}{d^2}
		\leq\frac4d
		\leq1.
		\]
		Thus equality holds throughout, and hence for each $\mu\in\mathcal E_1$, we have $\mu=\pm\frac12$. However, $4\mid\tau$ forces $T_\tau(\mu)\neq-1$, a contradiction. Hence $4\nmid\tau$. Since $\tau$ is even, $\tau\equiv2\pmod4$.
	\end{proof}
	
	\begin{lemma}\label{valency_dihedral}
		Let $n\geq3$, and let $S\subseteq D_n\setminus\{1\}$ satisfying $S=S^{-1}$ and $\langle S\rangle= D_n$. If $\cay(D_n,S)$ exhibits perfect state transfer, then $d\in\{2,4\}$.
	\end{lemma}
	\begin{proof}
		Suppose perfect state transfer occurs from $1$ to $u$ for some $u\in D_n\setminus\{1\}$ at the minimum time $\tau$. If $u=ba^j$, then Lemma~\ref{valency_reflection} gives the result. If $u=a^j$, then by Theorems~\ref{pst_dn_1}--\ref{pst_dn_4}, we must have $n=2m$ and $j=m$. 
		
		We first show that $d$ is even. Suppose that $d$ is odd.  Then $(2d_a-d)/d\notin\{1/2,0,-1/2\}$. Therefore by Lemma~\ref{algebraic_pst},
		\[
		\widetilde\mu_2=\frac{2d_a-d}{d}
		\in\left\{\pm1\right\}.
		\]
		Connectedness excludes $d_b=0$, and hence $d_a=0$. Applying Lemma~\ref{algebraic_pst} again, we obtain $\widetilde\mu_3\in\{\pm1\}$, and hence either $d_{b,0}=0$ or $d_{b,1}=0$. That is, all elements of $S_b^*$ have the same parity. Thus $S$ is contained in $\langle a^2,ba^k\rangle$ for some $k\in\{0,1\}$, contradicting connectedness. Therefore $d$ is even.

		Suppose that $d$ is even and $d>4$. By Lemma~\ref{central_time}, $\tau\equiv2\pmod4$, and hence	$T_\tau(0)=-1$.
		Since $T_\tau(\mu)=1$ for $\mu\in\mathcal E_0$, each element of $\mathcal E_0$ is nonzero.
		By \eqref{equality_1334},
		\[
		1\leq | \nu_0|^{1/m}
		\leq\frac{4(d+c)}{d^2}
		\leq\frac8d,\quad\text{as } c\leq d.
		\]
		Hence $d\leq8$. If $d=8$, then equality holds throughout, and hence $\mu\in\{\pm1/2\}$ for each $\mu\in \mathcal E_0$, which contradicts that $\widetilde\mu_1=1$ and $\widetilde\mu_1\in\mathcal E_0$. Thus $d<8$, and consequently $d=6$. Consider $\widetilde\mu_2={(d_a-3)}/{3}$. Lemma~\ref{algebraic_pst} gives $d_a\in\{0,3,6\}$. Recall that $\widetilde{\mu}_2\in\mathcal E_0$. However, $d_a=3$ gives $\widetilde{\mu}_2=0$, a contradiction. Further, if $d_a=6$ then the graph becomes disconnected. Hence, we must have $d_a=0$, giving $d_b=6=d$. For $h\in\Zl_n$, set
		\[
		\theta_h:=\frac{|\eta_h(S_b)|^2}{9}
		=\frac19\sum_{k=0}^{n-1}\Delta_{S_b^*}(k)\omega_n^{hk}.
		\]
		Since $d_a=0$ and $d_b=6$, for $1\leq h\leq m-1$, we find  $\theta_h=(2\mu_h^{(1)})^2$ if $\eta_h(S_b)\neq0$, and $\theta_h=0$ otherwise. 	Also, $\theta_0=4$, $\theta_m=(2\widetilde\mu_3)^2$, and $\theta_{n-h}=\theta_h$. Hence, by Lemma~\ref{algebraic_pst}, $\theta_h$ is an algebraic integer for each $h\in\Zl_n$. Apply Lemma~\ref{divisibility} with $f=\Delta_{S_b^*}$ and $q=9$. If $3\nmid n$, part~(i) of Lemma~\ref{divisibility} gives $9\mid \Delta_{S_b^*}(0)=6$, a contradiction. If $3\mid n$, then $0\leq \Delta_{S_b^*}(k)\leq6<9$, and hence part~(ii) of Lemma~\ref{divisibility} gives $\widehat{f}(2)=0$. Since $\widehat f(h)=\sum_{k=0}^{n-1}\Delta_{S_b^*}(k)\omega_n^{hk}=|\eta_h(S_b)|^2$, we have $\eta_2(S_b)=0$. Therefore, $\mu_2=\eta_2(S_a)/d=0$. as $S_a=\emptyset$. Hence, $T_\tau(\mu_2)=T_\tau(0)=-1$. It contradicts that $T_\tau(\mu_2)=1$ (refer Theorems~\ref{pst_dn_1} and \ref{pst_dn_2}). Thus $d\in\{2,4\}$.
	\end{proof}

	\subsection{Classification of the connection sets}
	We now determine the possible connection sets, considering odd and even $n$ separately.
	\begin{lemma}\label{pst_odd_n}
		Let $n$ be odd with $n\geq3$, and let $S\subseteq D_n\setminus\{1\}$ satisfying $S=S^{-1}$ and $\langle S\rangle= D_n$. If $\cay(D_n,S)$ exhibits perfect state transfer, then one of the following holds.
		\begin{enumerate}
			\item[(i)] $S=\{ba^i,ba^j\}$, where $\gcd(n,i-j)=1$.
			\item[(ii)] $S=\{a^j,a^{-j},ba^{c-j},ba^{c+j}\}$, where $\gcd(n,j)=1$.
		\end{enumerate}
	\end{lemma}
	\begin{proof}
		Suppose that $\cay(D_n,S)$ exhibits perfect state transfer. By Lemma~\ref{valency_dihedral}, we have $d\in\{2,4\}$. If $d=2$, then by $S=S^{-1}$ and $\langle S\rangle= D_n$, it follows that $S=\{ba^i,ba^j\}$. Since $\langle S\rangle=\langle a^{j-i},ba^i\rangle$, we have $\gcd(n,i-j)=1$.
		
		Now suppose that $d=4$. Since $n$ is odd, Theorem~\ref{pst_dn_4} excludes the normal case. By vertex-transitivity and Theorem~\ref{pst_dn_3}, we may assume that perfect state transfer occurs from $1$ to $ba^c$, and $S_b^*=2c-S_b^*$.
		Define $C=S_b^*-c$ and  $f(k)=\mathbf1_{S_a^*}(k)+\mathbf1_C(k)$. As in the proof of Lemma~\ref{valency_reflection},
		$\widehat f(h)/2$ is an algebraic integer for every $h\in\Zl_n$. Since $n$ is odd, Lemma~\ref{divisibility}(i), with $q=2$, gives $2\mid f(k)$ for every $k\in\Zl_n$. Thus $f(k)\in\{0,2\}$, which implies $S_a^*=C=S_b^*-c$. Consequently, $d_a=d_b=2$. Since $S_a=S_a^{-1}$ and $n$ is odd, we may write $S_a=\{a^j,a^{-j}\}$. It follows that $S_b^*=\{c-j,c+j\}$, and hence $S=\{a^j,a^{-j},ba^{c-j},ba^{c+j}\}$. Finally, since $\langle S\rangle=\langle a^j,ba^c\rangle$, connectedness gives $\gcd(n,j)=1$.
	\end{proof}

	\begin{lemma}\label{pst_even_n}
		Let $n=2m\geq4$, and let $S\subseteq D_n\setminus\{1\}$ satisfying $S=S^{-1}$ and $\langle S\rangle= D_n$. If $\cay(D_n,S)$ exhibits perfect state transfer, then one of the following holds.
		\begin{enumerate}
			\item[(i)] $S=\{ba^i,ba^j\}$, where $\gcd(n,i-j)=1$.
			\item[(ii)] $S=\{a^j,a^{-j},ba^{c-j},ba^{c+j}\}$, where $4\nmid n$ and $\gcd(n,j)=1$. 
			\item[(iii)]  $S=\{ba^c,ba^{c+m},ba^{c+j},ba^{c+j+m}\}$, where $m$ is odd and  $\gcd(m,j)=1$.
		\end{enumerate}
	\end{lemma}
	\begin{proof}
		Let $\cay(D_n,S)$ exhibit perfect state transfer. By Lemma~\ref{valency_dihedral}, $d\in\{2,4\}$. If $d=2$, then clearly $S=\{ba^i,ba^j\}$ and  $\gcd(n,i-j)=1$.

		Now suppose that $d=4$. Let $\tau$ be the minimum time for exhibiting perfect state transfer. By Corollary~\ref{minimum} and Lemma~\ref{kubota_regular}, $\tau$ is even. In particular, $T_\tau(\pm1)=1$ and $T_\tau(\pm1/2)\in\{1,-1/2\}$. By Theorems~\ref{pst_dn_1} and~\ref{pst_dn_2}, perfect state transfer occurs from $1$ to either $a^m$ or $ba^c$ for some $c$. 
		
		\noindent\textbf{Case 1.} Perfect state transfer occurs from $1$ to $ba^c$. Then $T_\tau(\widetilde\mu_2)=-1$. By Lemma~\ref{algebraic_pst}, $\widetilde\mu_2\in\{\pm1,\pm1/2,0\}$, and hence $\widetilde\mu_2=0$, which gives $d_a=d_b=2$. 
		Write $S_a=\{a^j,a^{-j}\}$. Put $C=S_b^*-c$ and $f(k)=\mathbf1_{S_a^*}(k)+\mathbf1_C(k)$. By Theorem~\ref{pst_dn_1}(B), $C=-C$.
		As in the proof of Lemma~\ref{valency_reflection}, $\widehat f(h)/2$ is an algebraic integer for every $h\in\Zl_n$, and $\widehat f(1)\neq0$. If $C\neq S_a^*$, then $C\cap S_a^*=\emptyset$, and hence $f(k)\in\{0,1\}$. Therefore Lemma~\ref{divisibility}(ii), with $q=2$, gives $\widehat f(1)=0$, a contradiction.
		Thus $C=S_a^*$, giving $S=\{a^j,a^{-j},ba^{c-j},ba^{c+j}\}$ and connectedness gives $\gcd(n,j)=1$.
		If $m$ were even, then $j$ would be odd and $\eta_{n/4}(S_b)=2\omega_n^{cn/4}\cos\frac{\pi j}{2}=0$, contrary to Theorem~\ref{pst_dn_1}(B). Therefore $m$ is odd, and hence $4\nmid n$.

		\noindent\textbf{Case 2.} Perfect state transfer occurs from $1$ to $a^m$. By Lemma~\ref{central_time},
		$\tau\equiv2\pmod4$, so $T_\tau(0)=-1$. By Theorem~\ref{pst_dn_1}(A), $T_\tau(\widetilde\mu_2)=1$, we have $d_a\neq2$,
		while connectedness excludes $d_a=4$.
		
		Suppose that $d_a\in\{1,3\}$. Then $d_b\in\{1,3\}$ and $a^m\in S_a$. 
		Since $4=d=d_{a,0}+d_{a,1}+d_{b,0}+d_{b,1}$,  
		\begin{align*}
			\widetilde\mu_3=\frac{d_{a,0}-d_{a,1}+d_{b,0}-d_{b,1}}{4}&=\frac{d_{a,0}-d_{a,1}+d_{b,0}-d_{b,1}+(d_{a,0}+d_{a,1}+d_{b,0}+d_{b,1})-d}{4}\\&=\frac{d_{a,0}+d_{b,0}-2}{2}.
		\end{align*}
		Similarly, $\widetilde\mu_4=({d_{a,0}+d_{b,1}-2})/{2}$. Since $d_{b,0}$ and $d_{b,1}$ have opposite parity, one of these eigenvalues equals $\pm1/2$, and the other is an integer, say $k_0$. Then Theorem~\ref{pst_dn_1}(A) gives $(-1)^m=T_\tau(\pm1/2)\in\{1,-1/2\}$. Hence $m$ is even. Then $a^m\in S_a$ implies $d_{a,0}\geq1$, and the preceding formulas give $\widetilde\mu_3,\widetilde\mu_4\geq-1/2$. 	Connectedness also gives
		$\widetilde\mu_3,\widetilde\mu_4<1$. Hence $-1/2\leq k_0<1$, forcing $k_0=0$.
		Consequently, $1=(-1)^m=T_\tau(k_0)=T_\tau(0)=-1$, a contradiction.

		Therefore $d_a=0$ and $d_b=4$. That is, $|S_b^*|=4$.  Suppose that $\Delta_{S_b^*}(k)\leq3$ for $k\neq0$. Define $f(k)=\Delta_{S_b^*}(k)-4\mathbf1_{\{0\}}(k)$. Then $0\leq f(k)\leq3$. For $1\leq h\leq m-1$, we obtain
		\[\frac{\widehat f(h)}4=\frac{|\eta_h(S_b)|^2}{4}-1=\begin{cases}
			\left(2\mu_h^{(1)}\right)^2-1 &\text{if }\eta_h(S_b)\neq0\\
			-1&\text{if }\eta_h(S_b)=0.
		\end{cases}\]
		Similarly,
		\[\frac{\widehat f(0)}4=\frac{|\eta_0(S_b)|^2}{4}-1=3~\text{ and }~\frac{\widehat f(m)}4=\frac{|\eta_m(S_b)|^2}{4}-1=(2\widetilde\mu_3)^2-1.\]
		Therefore by  Lemma~\ref{algebraic_pst},  $\widehat f(h)/4$ is an algebraic integer for $0\leq h\leq m$. Since	$|\eta_{n-h}(S_b)|=|\eta_h(S_b)|$, $\widehat f(h)/4$ is an algebraic integer for each $h\in\Zl_n$. Then Lemma~\ref{divisibility}(ii), with $q=4$, gives $\widehat f(1)=0$. Hence $|\eta_1(S_b)|^2=4$, so $\mu_1^{(1)}=1/2$. Since $\tau$ is even, $T_\tau(\mu_1^{(1)})\neq-1$, contradicting Theorem~\ref{pst_dn_1}(A).

		Therefore, there must be some nonzero $k\in\Zl_n$ such that $\Delta_{S_b^*}(k)=4$. Then the order of $k$ divides $|S_b^*|$, and so the order of $k$ is either $2$ or $4$. Hence $k=m$ or $2k=m$, which gives $S_b^*+m=S_b^*$. Therefore $S_b^*=\{c,c+m,c+j,c+j+m\}$, giving that $S=\{ba^c,ba^{c+m},ba^{c+j},ba^{c+j+m}\}$. Since $\langle S\rangle=\langle a^m,a^j,ba^c\rangle$, connectedness gives $\gcd(m,j)=1$. If $m$ is even, then $j$ is odd, giving
		$d_{b,0}=d_{b,1}=2$. Hence $\widetilde\mu_3=0$.  Since $T_\tau(0)=-1$, contradicting  $T_\tau(\widetilde\mu_3)=(-1)^m=1$. Thus $m$ is odd.
	\end{proof}
	
	\subsection{Families exhibiting perfect state transfer}
	We verify perfect state transfer for three families of connection sets that appear in the classification.
	\begin{lemma}\label{pst_representatives}
		Let $n\geq3$. Then $\cay(D_n,S)$ exhibits perfect state transfer in each of
		the following cases.
		\begin{enumerate}
			\item[(i)] $S=\{b,ba\}$.
			\item[(ii)] $S=\{a,a^{-1},ba,ba^{-1}\}$, where $4\nmid n$.
			\item[(iii)] $S=\{b,ba^m,ba,ba^{m+1}\}$, where $n=2m$
			and $m$ is odd.
		\end{enumerate}
		The minimum time for perfect state transfer in \textup{(i)}, \textup{(ii)}, and \textup{(iii)} are $n$, $\lcm(n,2)$, and $2m$, respectively.
	\end{lemma}
	\begin{proof}
		We compute the eigenvalues of the discriminant $P$ using \eqref{ev_p_dn1}--\eqref{ev_p_dn3} and Theorems~\ref{ev1_dn} and~\ref{ev2_dn}.
		
		\noindent\textbf{(i)} Let $S=\{b,ba\}$. For $1\leq h\leq\lfloor(n-1)/2\rfloor$, we have $\eta_h(S_b)=1+\omega_n^h\neq0$. Then the eigenvalues of $P$ are $\widetilde{\mu}_1=1$, $\widetilde{\mu}_2=-1$, $\mu_h^{(1)}=\cos({\pi h}/{n})$, $\mu_h^{(2)}=-\cos({\pi h}/{n})$, with $\widetilde{\mu}_3=\widetilde{\mu}_4=0$ when $n$ is even.
		
		First suppose that $n=2m+1$. Choose $u=m\in G_0$ and $v=n\in G_1$ such that $2(v-u)=n+1$. Since $\ell_h=e^{-\iu\pi h/n}$, we obtain $S_b=a^{2(v-u)}S_b^{-1}$ and $\ell_h\omega_n^{(v-u)h}=(-1)^h$. Moreover, $T_n(\widetilde{\mu}_1)=-T_n(\widetilde{\mu}_2)=1$ and $T_n(\mu_h^{(1)})=-T_n(\mu_h^{(2)})=(-1)^h$.	Hence by Theorem~\ref{pst_dn_3}, $\cay(D_n,S)$ exhibits perfect state transfer
		from $u$ to $v$ at time $n$.
		
		Now suppose that $n=2m$. Let $(u,v)\in(G_0\times G_0)\cup(G_1\times G_1)$ with $u-v=\pm m$. Then we have $T_n(\widetilde{\mu}_1)=T_n(\widetilde{\mu}_2)=1$, $T_n(\widetilde{\mu}_3)=T_n(\widetilde{\mu}_4)=(-1)^m$ and $T_n(\mu_h^{(1)})=T_n(\mu_h^{(2)})=(-1)^h$.	Thus condition \textbf{(A)} of Theorem~\ref{pst_dn_1} holds, and hence $\cay(D_n,S)$ exhibits perfect state transfer from $u$ to $v$ at time $n$.

		\noindent\textbf{(ii)} Let $S=\{a,a^{-1},ba,ba^{-1}\}$, where $4\nmid n$. For $1\leq h\leq\lfloor(n-1)/2\rfloor$, set $q_h=\cos(2\pi h/n)$. Then $\eta_h(S_a)=\eta_h(S_b)=2q_h\neq0$ and  $\ell_h=\operatorname{sgn}(q_h)$. Therefore the eigenvalues of $P$ are $\widetilde{\mu}_1=1$, $\widetilde{\mu}_2=0$, $\mu_h^{(1)}=({q_h+|q_h|})/{2}$,   $\mu_h^{(2)}=({q_h-|q_h|})/{2}$.  When $n$ is even, the remaining one-dimensional eigenvalues are $\widetilde{\mu}_3=-1$ and $\widetilde{\mu}_4=0$.

		Let $\kappa=\lcm(n,2)$. Then $T_\kappa(q_h)=1$ and $T_\kappa(0)=-1$. Consequently, $T_\kappa(\widetilde{\mu}_1)=-T_\kappa(\widetilde{\mu}_2)=1$ and $T_\kappa(\mu_h^{(1)})=-T_\kappa(\mu_h^{(2)})=\operatorname{sgn}(q_h)$. Choose $u=0$ and $v=n$. Then $S_b=a^{2(v-u)}S_b^{-1}$ and $\ell_h\omega_n^{(v-u)h}=\operatorname{sgn}(q_h)$.

		If $n$ is odd, Theorem~\ref{pst_dn_3} shows that $\cay(D_n,S)$ exhibits perfect state transfer at time $\kappa$. If $n$ is even, we also have $T_\kappa(\widetilde{\mu}_3)=-T_\kappa(\widetilde{\mu}_4)=1=(-1)^{u+v}$. Thus condition \textbf{(B)} of Theorem~\ref{pst_dn_1} also gives perfect state transfer at time $\kappa$.

		\noindent\textbf{(iii)} Let $S=\{b,ba^m,ba,ba^{m+1}\}$, where $n=2m$ and $m$ is odd. For $1\leq h\leq m-1$, we have
		\[
		\eta_h(S_b)
		=(1+(-1)^h)(1+\omega_n^h)
		=
		\begin{cases}
			0&\text{if $h$ is odd},\\
			2(1+\omega_n^h)&\text{if $h$ is even}.
		\end{cases}
		\]
		Thus the eigenvalues of $P$ are
		\begin{align*}
			\widetilde{\mu}_1&=1,\quad
			\widetilde{\mu}_2=-1,\quad
			\widetilde{\mu}_3=\widetilde{\mu}_4=0,\\
			\mu_h&=0\quad\text{for odd $h$},\\
			\mu_h^{(1)}&=\cos\frac{\pi h}{n},
			\quad\text{and}\quad
			\mu_h^{(2)}=-\cos\frac{\pi h}{n}
			\quad\text{for even $h$}.
		\end{align*}
		Since $m$ is odd, we have $T_{2m}(0)=-1$, $T_{2m}(\pm1)=1$, and
		$T_{2m}\left(\pm\cos\frac{\pi h}{n}\right)=1$ for even $h$. Therefore, for $u=0$ and $v=m$, all the conditions in \textbf{(A)} of Theorem~\ref{pst_dn_1} hold at time $2m$.
		Hence $\cay(D_n,S)$ exhibits perfect state transfer at time $2m$.
		
		Finally, suppose $\tau$ is the minimum transfer time. By Theorem~\ref{ev_grover} and Lemma~\ref{periodic},
		the period $\sigma$ of $\cay(D_n,S)$ is
		\[
		\sigma=
		\begin{cases}
			2n & \text{in case \textup{(i)}},\\
			2\lcm(n,2) & \text{in case \textup{(ii)}},\\
			4m & \text{in case \textup{(iii)}}.
		\end{cases}
		\]
		Since perfect state transfer occurs in each case, Corollary~\ref{minimum} gives $\tau=\sigma/2$.
	\end{proof}
	
	\subsection{Proof of the complete classifications}
	We use the preceding results to prove the main theorem and then identify the corresponding graphs.
	\begin{proof}[\textbf{Proof of Theorem~\ref{main_pst_dn}}]
		For integers $r$ and $s$ with $\gcd(r,n)=1$, let $\varphi_{r,s}(a)=a^r$ and $\varphi_{r,s}(b)=ba^s$.
		Then $\varphi_{r,s}$ is an automorphism of $D_n$.
		
		Suppose that $\cay(D_n,S)$ exhibits perfect state transfer. Then Lemmas~\ref{pst_odd_n} and~\ref{pst_even_n} give the exact connection sets.
		
		First, suppose $S=\{ba^i,ba^j\}$ with $\gcd(n,j-i)=1$. Choose $r$ such that $r(j-i)\equiv1\pmod n$, and put $s=-ri$. Then $\varphi_{r,s}(S)=\{b,ba\}$.
		
		Next, suppose $S=\{a^j,a^{-j},ba^{c-j},ba^{c+j}\}$ with $4\nmid n$ and $\gcd(n,j)=1$. Choose $r$ such that $rj\equiv1\pmod n$ and put $s=-rc$. Then $\varphi_{r,s}(S)=\{a,a^{-1},ba^{-1},ba\}$. 
		
		Finally, suppose $S=\{ba^c,ba^{c+m},ba^{c+j},ba^{c+j+m}\}$, where $m$ is odd and $\gcd(j,m)=1$. Since $m$ is odd, exactly one of $j$ and $j+m$ is odd. Thus, replacing $j$ by $j+m$ if necessary, we may assume that $j$ is odd without changing $S$. Therefore $\gcd(j,2m)=1$. Choose $r$ such that $rj\equiv1\pmod{2m}$ and put $s=-rc$. Since $r$ is odd, $rm\equiv m\pmod{2m}$, and hence $\varphi_{r,s}(S)=\{b,ba^m,ba,ba^{m+1}\}$. This proves that $(D_n,S)$ is equivalent to one of the listed pairs. 
		
		Conversely, Lemma~\ref{pst_representatives} shows that the Cayley graphs in cases \textup{(i)}, \textup{(ii)}, and \textup{(iii)} exhibit perfect state transfer.

		It remains to show that the three pairs are inequivalent. The connection set in \textup{(i)} has two elements, while those in \textup{(ii)} and \textup{(iii)} have four. Since $a$ has order $n$ with $n\geq3$ and every reflection has order two, an automorphism must send $a$ to a generator of $\langle a\rangle$. Hence it preserves the number of rotations in a connection set. Case \textup{(ii)} contains two rotations, while case \textup{(iii)} contains none, so these pairs are also inequivalent.
	\end{proof}
	
	\begin{proof}[\textbf{Proof of Proposition~\ref{pst_transfer_pairs}}]
		The minimum transfer time follows from Lemma~\ref{pst_representatives}. For $g\in D_n$, the right multiplication map $\mathcal R_g\colon D_n\to D_n$, given by $\mathcal R_g(h)=hg$, is an automorphism of $\cay(D_n,S)$. It is known from  \cite[Lemma~3.4]{bhakta3} that $\mathcal R_g$ preserves perfect state transfer.
		
		In case \textup{(i)} with $n$ even and in case \textup{(iii)}, perfect state transfer occurs between $1$ and $a^{n/2}$. Right multiplication by $a^r$ and $ba^r$ gives the pairs $\{a^r,a^{r+n/2}\}$ and $\{ba^r,ba^{r+n/2}\}$, respectively. In case \textup{(ii)}, right multiplication of the transfer pair $\{1,b\}$ by $a^r$ gives $\{a^r,ba^r\}$. In case \textup{(i)} with $n$ odd, right multiplication of $\{a^{(n-1)/2},b\}$ by $a^{r-(n-1)/2}$ gives the transfer pair $\{a^r,ba^{r+(n+1)/2}\}$.
		
		As $r$ runs through $\Zl_n$, the listed pairs cover all possible vertices. Hence, by Corollary~\ref{pair_minimum}, the table lists all transfer pairs.
	\end{proof}
	\begin{proof}[\textbf{Proof of Corollary~\ref{pst_graphs_dn}}]
		We identify the graphs corresponding to the three connection sets in Theorem~\ref{main_pst_dn}. 
		
		First, $\cay(D_n,\{b,ba\})$ is connected and $2$-regular on $2n$ vertices. Hence it is isomorphic to the even cycle $C_{2n}$.
		
		Next, let $S=\{a,a^{-1},ba,ba^{-1}\}$. Partition the vertices into $V_i=\{a^i,ba^i\}$ for $i\in\mathbb Z_n$. Each $V_i$ is independent, and each vertex in $V_i$ is adjacent precisely to all the vertices in $V_{i-1}$ and $V_{i+1}$. Thus $\cay(D_n,S)$ is isomorphic to $C_n[\overline{K}_2]$.
		
		Finally, let $n=2m$ and $S=\{b,ba^m,ba,ba^{m+1}\}$. For $i\in\mathbb Z_m$, define $V_i=\{a^i,a^{i+m}\}$ and $W_i=\{ba^i,ba^{i+m}\}$. These are independent pairs of vertices. Every vertex in $V_i$ is adjacent precisely to all the vertices in $W_i$ and $W_{i+1}$. Therefore $\cay(D_{2m},S)$ is isomorphic to  $C_{2m}[\overline{K}_2]$. Thus the result follows.
	\end{proof}

	\section*{Acknowledgments}
	The author Koushik Bhakta gratefully acknowledges financial support from the Prime Minister’s Research Fellowship (PMRF), Government of India (PMRF-ID: 1903298). Xiwang Cao's work is supported by the NNSF of China (No. 12571575).

\end{document}